\documentclass[11pt]{amsart}
\usepackage{amsmath,amscd,amssymb,amsfonts,amsthm}

\usepackage[cmtip,matrix,arrow]{xy}

\newtheorem{theorem}{Theorem}[section]

\newtheorem{proposition}[theorem]{Proposition}

\newtheorem{lemma}[theorem]{Lemma}
\theoremstyle{definition}    
\newtheorem{definition}[theorem]{Definition}
\theoremstyle{remark}

\newtheorem{remark}[theorem]{Remark}
\newtheorem{remarks}[theorem]{Remarks}
\newtheorem{example}[theorem]{Example}

\newcommand\A{\mathcal{A}}

\newcommand\M{\mathcal{M}}
\newcommand\G{\mathcal{G}}
\newcommand{\K}{\mathsf{K}}

\newcommand{\Spin}{\on{Spin}}
\renewcommand{\O}{\mathcal{O}}
\newcommand{\T}{\mathbf{T}}

\newcommand{\Co}{\mathcal{C}}
\newcommand{\ca}{\mathcal}

\newcommand{\U}{\on{U}}

\newcommand{\N}{\mathcal{N}}
\newcommand{\R}{\mathbb{R}}
\newcommand{\C}{\mathbb{C}}
\newcommand{\SU}{\on{SU}}

\newcommand{\Z}{\mathbb{Z}}

\renewcommand{\P}{\mathcal{P}}
\newcommand{\PG}{\ca{P} G}
\newcommand{\sz}{\mathsf{s}}
\newcommand{\Cl}{\C l}
\newcommand\lie[1]{\mathfrak{#1}}

\newcommand{\g}{\lie{g}}

\renewcommand{\t}{\lie{t}}

\newcommand{\on}{\operatorname}

\newcommand{\Ad}{ \on{Ad} }
\newcommand{\ad}{\on{ad}}
\newcommand{\Hom}{ \on{Hom}}

\renewcommand{\subset}{\subseteq}

\newcommand{\Sz}{\mathsf{S}}
\newcommand{\s}{{(s)}}
\renewcommand{\ker}{ \on{ker}}

\newcommand{\cox}{\mathsf{h}^\vee}
\newcommand\qu{/\kern-.7ex/} 
\newcommand{\fus}{\circledast} 
\newcommand{\Waff}{W_{\text{aff}}} 

\newcommand{\hra}{\hookrightarrow}

\renewcommand{\d}{{\mbox{d}}}
\newcommand{\ol}{\overline}

\newcommand\Phinv{\Phi^{-1}}

\newcommand\eps{\epsilon}

\newcommand\om{\omega}

\newcommand{\f}{\frac}

\newcommand{\p}{\partial}
\renewcommand{\l}{\langle}
\renewcommand{\r}{\rangle}
\newcommand\hh{{\f{1}{2}}}

\newcommand{\eeq}{\end{eqnarray*}}
\newcommand{\beq}{\begin{eqnarray*}}

\newcommand{\Hol}{\on{Hol}}

\newcommand{\wh}{\widehat}
\newcommand{\wt}{\widetilde}
\newcommand{\mf}{\mathfrak}

\newcommand{\n}{\mf{n}}

\newcommand\dirac{/\kern-1.2ex\partial} 

\renewcommand{\i}{\mathsf{i}}
\newcommand{\tpi}{{2\pi\i}}

\newcommand{\ev}{\on{ev}}
\begin{document}
\sloppy
\title{Spinor modules for Hamiltonian loop group spaces}
\author{Yiannis Loizides}
\author{Eckhard Meinrenken} 
\author{Yanli Song}

\begin{abstract}
Let $LG$ be the loop group of a compact, connected Lie group $G$. We show that the tangent bundle of any proper Hamiltonian $LG$-space $\M$ has a natural completion $\ol{T}\M$ to a \emph{strongly} symplectic $LG$-equivariant vector bundle. This bundle admits an invariant  
compatible complex structure within a natural polarization class, defining an $LG$-equivariant spinor bundle $\Sz_{\ol{T}\M}$, which one may regard as the $\Spin_c$-structure of $\M$. We  
describe two procedures for obtaining a finite-dimensional version of this spinor module. In one approach, we construct  from $\Sz_{\ol{T}\M}$ a twisted $\Spin_c$-structure for the quasi-Hamiltonian $G$-space associated to $\M$. In the second approach,  we describe an `abelianization procedure', passing to a finite-dimensional $T\subset LG$-invariant submanifold of $\M$, and we show how to construct an equivariant $\Spin_c$-structure on that submanifold. 
\end{abstract}
\maketitle
\section{Introduction}
Let $G$ be a compact, connected Lie group, with an invariant inner product $\cdot $ on $\g$.
We take the loop group $LG$ to be the Banach Lie group of $G$-valued loops of a fixed Sobolev class $\sz>\f{1}{2}$;  loops in this Sobolev range are continuous, and the group structure is given by pointwise multiplication. 
The loop group acts  by gauge transformations on the space $\A$ of 
connections over the circle,  given as $\g$-valued 1-forms on $S^1$ of Sobolev class $s-1$. 
\[ \lambda.\mu=\Ad_\lambda\mu-\partial \lambda\ \lambda^{-1},\ \ \ \mu\in \A,\ \ \lambda\in LG.\] 
A \emph{proper Hamiltonian loop group space} $(\M,\omega,\Phi)$ is a weakly symplectic Banach manifold $\M$, with an action of the loop group and a proper $LG$-equivariant \emph{moment map}
\[ \Phi\colon \M\to \A\]
satisfying the moment map condition $\iota(\xi_\M)\omega=-\d\l\Phi,\xi\r$ for $\xi\in L\g$, where the function $\l\Phi,\xi\r$ is defined by the pointwise inner product of $\Phi$ and $\xi$, followed by integration over $S^1$.

The 2-form $\omega$ being  \emph{weakly} symplectic means that the bundle map $\omega^\flat \colon T\M\to T^*\M$ is injective, in contrast to \emph{strongly} symplectic 2-forms for which it is required to be an isomorphism. Our first result is 
\begin{theorem}
The tangent bundle of any proper Hamiltonian loop group space $\M$ has a canonically defined $LG$-equivariant completion $\ol{T}\M$, such that 
the 2-form $\omega$  extends to a \emph{strongly} symplectic 2-form on $\ol{T}\M$. 
\end{theorem}
Roughly speaking,  this completion $\ol{T}\M$ is obtained by taking the Sobolev $\f{1}{2}$ completion in orbit directions -- note that this is precisely the borderline case where $G$-valued loops no longer form a group. 

To define a spinor module, the next step is to choose a compatible complex structure. For finite-dimensional symplectic manifolds, the  resulting spinor module does not depend on the choice, up to isomorphism. In infinite dimensions, the situation is more delicate \cite{ply:sp}. We will consider equivalence classes of complex structures, where two complex structures are equivalent if their difference  is Hilbert-Schmidt. Such an equivalence class is sometimes called a \emph{polarization}. 
\begin{theorem}
The bundle $\ol{T}\M$ has a distinguished  $LG$-invariant polarization. It admits a global 
 $LG$-invariant $\omega$-compatible complex structure $J$ within this polarization class, unique up to homotopy. 
\end{theorem}
The Riemannian metric on $\ol{T}\M$ associated to $J$ and $\omega$ defines a bundle of Clifford algebras, $\Cl(\ol{T}\M)$, and using $J$ one obtains a $\Z_2$-graded spinor bundle, 
\[ \Cl(\ol{T}\M)\circlearrowright \Sz_{\ol{T}\M}.\]
For finite-dimensional Riemannian manifolds of even dimension, a $\Z_2$-graded spinor module over the Clifford bundle is the same thing as a $\Spin_c$-structure; hence we may think of $ \Sz_{\ol{T}\M}$ as defining a $\Spin_c$-structure on $\M$. Given a  pre-quantum line bundle $\ca{L}\to \M$, then one can form a new $\Spin_c$-structure $\ca{L}\otimes \Sz_{\ol{T} \M}$, and the action of a
$\Spin_c$-Dirac operator of $\M$ should formally describe the `quantization' of $\M$. In the case of moduli spaces of flat connections, a technique for constructing pre-quantum line bundles $\ca{L} \to \M$ was explained in \cite{rsw:cs}, in the context of geometric quantization for Chern-Simons theory.

In practice, dealing with Dirac operators in infinite dimensions is too difficult, and one prefers 
to work with suitable finite-dimensional counterparts. As shown in \cite{al:mom}, every 
proper Hamiltonian $LG$-space $(\M,\omega,\Phi)$ has an associated quasi-Hamiltonian $G$-space $(M,\omega_M,\Phi_M)$, with a group-valued moment map $\Phi_M\colon M\to G$. In this paper, we will interpret this relation using a correspondence diagram 
\[ \xymatrix{ & \ca{N}\ar[dl]_p\ar[dr]^q & \\
\M & & M
}\]
Here $\N$ is a Banach manifold with an action of $LG\times G$, where the $G$-action is principal with quotient map $p$, and the $LG$-action is principal as well, with quotient map $q$.   The idea is to pull the spinor module $\Sz_{\ol{T}\M}$ back to $\N$, and then 
push forward to $M$. For the second step, one uses the spinor module for the 
Lie algebra of the loop group, in order to quotient out the $LG$-orbit directions. This spinor module is equivariant with respect to a central extension of the loop group  $LG$ by $\U(1)$, where the central circle acts non-trivially. Hence, the construction does not quite give a 
spinor module over $TM$: 
\begin{theorem}
The bundle $q^*TM\to \N$ has an $\wh{LG}^{\on{spin}}\times G$-equivariant $\Spin_c$-structure, with associated spinor module 
\[ \Cl(q^*TM)\circlearrowright \Sz_{q^*TM}.\] 
Here $\wh{LG}^{\on{spin}}$ is the \emph{spin central extension} of the loop group. This $\Spin_c$-structure on $q^*TM$ is canonical, up to equivariant homotopy. 
\end{theorem}
We think of  $\Sz_{q^*TM}$ as a \emph{twisted} $\Spin_c$-structure on $M$, in the spirit of Murray-Singer \cite{mur:ger} and Mathai-Melrose-Singer \cite{mat:fra}. In the pre-quantized case, one uses the pre-quantum line bundle $\ca L$ to define a new twisted 
$\Spin_c$-structure $p^*\ca{L}\otimes \Sz_{q^*TM}$. The twisted $\Spin_c$structure may also be interpreted in terms of a 
$G$-equivariant Morita morphism from the Clifford bundle $\Cl(TM)$ over $M$ to a Dixmier-Douady bundle $\mathsf{A}^{\on{spin}}$ over $G$; see \cite{al:ddd,me:twi}. 
\smallskip


Another finite-dimensional approach is the following abelianization procedure. Let $T\subset G$ be a maximal torus, with Lie algebra $\t$, and fix a system of positive roots of $(G,T)$. The integral lattice  $\Lambda\subset \t$ can be regarded as a subgroup of $LG$, consisting of exponential loops, while $N(T)\subset G$ 
is a subgroup of $LG$ consisting of constant loops. The central extension of $LG$ restricts to the subgroup $\Lambda\rtimes N(T)$.


Suppose the moment map $\Phi\colon \M\to \A$ is transverse to $\t$. Then $\ca{X}=\Phi^{-1}(\t)$ is a finite-dimensional pre-symplectic manifold, with a Hamiltonian action of $\Lambda\times T$, with an equivariant 
moment map  $\Phi_{\ca{X}}\colon \ca{X}\to \t$ where $\Lambda$ acts on $\t$ by translation. 

\begin{theorem}\label{th:3}
The $LG$-equivariant $\Spin_c$-structure on $\M$ determines a 
$\Spin_c$-structure 
\[ \Cl(T\ca{X})\circlearrowright \Sz_{T\ca{X}}\]
on $\ca{X}$, equivariant under the action of the spin-central extension of $\Lambda\rtimes N(T)$. Up to homotopy, the spinor module $ \Sz_{T\ca{X}}$ depends only on the choice  of positive roots.    
\end{theorem}
In the pre-quantized case, one considers the Dirac operator associated to the new spinor module $\ca{L}|_{\ca{X}}\otimes \Sz_{T\ca{X}}$; it is equivariant for the action of a 
semi-direct product of $T$ with a (different) central extension $\wh{\Lambda}$ of the lattice. In forthcoming work, we will show that the associated Dirac operator over the 
non-compact manifold $\ca{X}$ has a well-defined $T$-equivariant index, with finite multiplicities. This crucially depends on the $\wh{\Lambda}$-equivariance. We 
will show furthermore how to compute this index using localization for the norm-square of the moment map. In the case of a compact Hamiltonian $G$-space, an analogous formula was proved by Paradan \cite{par:rr}.  In the loop group setting, the norm-square of the moment map has been used to prove a Kirwan surjectivity theorem \cite{btw:kir}, and to study (twisted) Duistermaat-Heckman distributions \cite{loi:nor}.

The organization of this article is as follows. Section \ref{sec:infinity} starts with a review of spinor modules in infinite dimensions, recalling the classical result that the isomorphism class of such a spinor module defined by a complex structure depends on the polarization class of the complex structure. Given a symplectic form, we show that any two 
compatible complex structures in a given polarization class are homotopic within that polarization class. The subsequent section \ref{sec:ham} explains the relationship between Hamiltonian loop group spaces and quasi-Hamiltonian spaces as a Morita equivalence. Section \ref{sec:spin} constructs the spinor bundle for a Hamiltonian loop group space, and Section \ref{sec:twist} gives the twisted $\Spin_c$-structure for the associated quasi-Hamiltonian space. Section \ref{sec:abel} is concerned with the abelianization procedure for the transverse case, along with a discussion of how to adjust this procedure for the possibly non-transverse case.  The final Section \ref{sec:twisted} explains how to generalize all these constructions to the case of twisted loop groups (where the twist is by an automorphism $\kappa\in \on{Aut}(G)$) and the associated $\kappa$-twisted quasi-Hamiltonian spaces.\\

\smallskip

\noindent \textbf{Acknowledgements.} We are grateful to Nigel Higson and Tudor Ratiu for helpful discussions.

\section{Spinor modules in infinite dimensions}\label{sec:infinity}
The theory of spinor modules for infinite-dimensional real Hilbert spaces $\mathsf{H}$ was developed in the 1960s by Shale and Stinespring, Araki, and others. References for this section include the books by Plymen-Robinson \cite{ply:sp} and Pressley-Segal \cite{pr:lo}; see also Freed-Hopkins-Teleman \cite[Section 3.1]{fr:lo2}.  

In this section, we will review some of the theory, with its application to the construction of the spin representation of the loop group. We will also need to discuss the setting where the initial data given on $\mathsf{H}$ is a symplectic structure rather than a Riemannian metric, with the latter depending on a choice of a compatible complex structure. 

\subsection{Notation}
Recall that a topological vector space $\mathsf{H}$ is \emph{banachable} (resp. \emph{hilbertable}) if its topology may be defined by a Banach norm (resp. Hilbert inner product) on $\mathsf{H}$. We will use the simpler terminology \emph{Banach space} (resp. Hilbert space), keeping in mind that we do not consider the norm (resp. Hilbert metric) to be part of the structure. For (real or complex) Banach spaces $\mathsf{H}_1,\mathsf{H}_2$, we denote by $\mathbb{B}(\mathsf{H}_1,\mathsf{H}_2)$ the Banach space of continuous linear operators from $\mathsf{H}_1$ to $\mathsf{H}_2$, and by $\mathbb{K}(\mathsf{H}_1,\mathsf{H}_2)$ the compact operators (the limits of finite rank operators). If $\mathsf{H}_1=\mathsf{H}_2=\mathsf{H}$ we write $\mathbb{B}(\mathsf{H})$ and $\mathbb{K}(\mathsf{H})$ for the algebras of bounded and compact operators, respectively.   If $\mathsf{H}_1,\mathsf{H}_2$ are separable 
Hilbert spaces, we write $\mathbb{B}_{HS}(\mathsf{H}_1,\mathsf{H}_2)\cong \mathsf{H}_2\otimes \mathsf{H}_1^*$ (using the Hilbert space tensor product) for the space of Hilbert-Schmidt operators from $\mathsf{H}_1$ to $\mathsf{H}_2$, with the notation $\mathbb{B}_{HS}(\mathsf{H})$ if $\mathsf{H}_1=\mathsf{H}_2=\mathsf{H}$. (This does not depend on the choice of metric, only on the topology.)  Introduce an equivalence relation on bounded operators $A\in \mathbb{B}(\mathsf{H})$, where 
\begin{equation}\label{eq:relation}
 A_0\sim A_1\ \ \ \Leftrightarrow \ \ \ A_0-A_1\in \mathbb{B}_{HS}(\mathsf{H}).
\end{equation}
More generally, we have such an equivalence relation on the space $\mathbb{B}(\mathsf{H}_1,\mathsf{H}_2)$ of bounded linear operators between possibly different Hilbert spaces. 

Let $\ca{J}(\mathsf{H})\subset \mathbb{B}(\mathsf{H})$ be the set of complex structures on $\mathsf{H}$, that is, $J^2=-\on{id}$. Following \cite{fr:lo2,pr:lo}  
we define
\begin{definition}
A \emph{polarization} of a real Hilbert space $\mathsf{H}$ is an equivalence class of complex structures $J\in \ca{J}(\mathsf{H})$, using the equivalence relation \eqref{eq:relation}. We denote by $\ca{J}_{\on{res}}(\mathsf{H})$ the set of complex structures in the given polarization class. 
\end{definition}
%

\subsection{Spinor modules}\label{subsec:spinor}
Let $\mathsf{H}$ be a separable real Hilbert space, and $g$ a Riemannian metric on $\mathsf{H}$, in the strong sense that the map $g^\flat \colon \mathsf{H}\to \mathsf{H}^*$ is an isomorphism.  We denote by $\on{O}(\mathsf{H})=\on{O}(\mathsf{H},g)\subset \mathbb{B}(\mathsf{H})$ the corresponding orthogonal group, and by $\Cl(\mathsf{H})=\Cl(\mathsf{H},g)$ the Clifford algebra, i.e. the complex algebra linearly generated by $\mathsf{H}$, with relations 
\[ v_1v_2+v_2v_1=2g(v_1,v_2),\ \ \ v_1,v_2\in \mathsf{H}.\] 
Given an orthogonal complex structure $J\in \on{O}(\mathsf{H})$ (that is, $J$ preserves $g$ and squares to $-\on{id}$), the complexified Hilbert space splits as 
$\mathsf{H}^\C=\mathsf{H}^+\oplus \mathsf{H}^-$, where $\mathsf{H}^\pm$ are the $\pm\sqrt{-1}$ eigenspaces of $J$. 
The \emph{restricted orthogonal group} $\on{O}_{res}(\mathsf{H})$ consists of those 
orthogonal transformations 
$A\in \on{O}(\mathsf{H})$ that preserve the polarization, i.e.,  such that $AJA^{-1}\sim J$. Let $\wedge \mathsf{H}^+$ be the exterior algebra over $\mathsf{H}^+$. Its Hilbert space completion  
\[ \Sz_\mathsf{H}=\ol{\wedge \mathsf{H}^{+}}\]
has the structure of a $\Z_2$-graded unitary $\Cl(\mathsf{H})$-module, where the Clifford action $\varrho(v)$ of elements  
$v\in \mathsf{H}^+$ is by exterior multiplication and that of $v\in \mathsf{H}^-$ is contraction. It will be referred to as the \emph{spinor module} defined by $J$. We sometimes write $\Sz_{\mathsf{H},J}$ to indicate the dependence on $J$.

The following theorem summarizes results of Araki and Shale-Stinespring, see \cite[Chapter 3]{ply:sp} 
\begin{theorem}[Spinor modules in infinite dimensions]\label{th:shasp}  
\vskip.0in \ \

\begin{enumerate}
\item 
The spinor modules defined by two orthogonal complex structures $J_0,J_1$ are isomorphic as ungraded $\Cl(\mathsf{H})$-modules if and only if they define the same polarization, i.e. $J_0\sim J_1$.  
\item 
In this case, the space 
$L=\on{Hom}_{\Cl(\mathsf{H})}(\Sz_{\mathsf{H},J_0},\Sz_{\mathsf{H},J_1})$ of intertwining operators is 1-dimensional, and $\Sz_{\mathsf{H},J_1}=L\otimes \Sz_{\mathsf{H},J_0}$. For $J_0\sim J_1$, the kernel of $J_0+J_1$ is even dimensional, and the parity of $L$ coincides with the parity of $\hh \dim \ker(J_0+J_1)$. 
\item 
An orthogonal transformation $A\in \on{O}(\mathsf{H})$ admits a unitary implementer $\wh{A}\in \on{U}(\Sz_{\mathsf{H},J})$, i.e. 
$\varrho(Av)=\wt{A}\varrho(v) \wt{A}^{-1}$ for $v\in \mathsf{H}$, if and only if $A\in \on{O}_{res}(\mathsf{H})$. In this case, the implementer is unique up to scalar. 
\end{enumerate}
\end{theorem}
Part (c) defines a central extension of the restricted orthogonal group 
$\on{O}_{res}(\mathsf{H})$ by $\U(1)$; this can be taken as a definition of the 
group $\on{Pin}_c(\mathsf{H})$ with respect to the choice of polarization. 

We will also need the following addendum: 
%
\begin{proposition}\label{prop:subspace}
 Let $W\subset \mathsf{H}$ be a finite-dimensional subspace of even dimension, and $\mathsf{H}'=W^\perp$ its orthogonal space with respect to $g$. Let $J,J'$ be orthogonal complex structures on $\mathsf{H},\mathsf{H}'$. Then 
\[ \Sz_W=\Hom_{\Cl(\mathsf{H}')}(\Sz_{\mathsf{H}'},\Sz_\mathsf{H})\]
is non-trivial if and only if $J\sim J'$, and in this case it is a spinor module over $\Cl(W)$. 
\end{proposition}
In this statement, $J'\in \mathbb{B}(\mathsf{H}')$ is regarded as an operator  on $\mathsf{H}$, acting as zero on $W$. 
\begin{proof}
Pick an auxiliary orthogonal complex structure $J_W$ on $W$. Then $J''=J_W\oplus J'$ is an orthogonal complex structure on $\mathsf{H}=W\oplus \mathsf{H}'$. Since $J''\sim J'$ we have that $\Sz_{J''}=\Sz_{J_W}\otimes \Sz_{J'}$, hence 
\[ \Hom_{\Cl(\mathsf{H}')}(\Sz_{\mathsf{H}',J'},\Sz_{\mathsf{H},J''})=\Sz_{W,J_{W}}.\]
By Theorem \ref{th:shasp}, $L=\Hom_{\Cl(\mathsf{H})}(\Sz_{\mathsf{H},J''},\Sz_{\mathsf{H},J})$ is non-trivial if and only if $J\sim J''\sim J'$, and in this case $\Sz_{\mathsf{H},J}=\Sz_{\mathsf{H},J''}\otimes L$, hence 
\[ \Sz_W=\Hom_{\Cl(\mathsf{H}')}(\Sz_{\mathsf{H}',J'},\Sz_{\mathsf{H},J})=\Sz_{W,J_{W}}\otimes L.\qedhere\]
\end{proof}
We will need to understand the dependence  on the metric. If $\mathsf{H}$ is finite-dimensional, then any two metrics $g_0,g_1$ on $\mathsf{H}$, with associated maps $g_i^\flat\colon \mathsf{H}\to \mathsf{H}^*$ and $g_i^\sharp=(g_i^\flat)^{-1}$, are related by 
an isometry $A=(g_1^\sharp\circ g_0^\flat)^{1/2}$, in the sense that 
$g_1(Av,Aw)=g_0(v,w)$. In the infinite-dimensional case, we need further conditions to ensure that the square root is defined. Introduce an equivalence relation on Riemannian metrics on $\mathsf{H}$, by declaring that $g_0\sim g_1$ if and only if 
$g_0^\flat\sim g_1^\flat$, i.e., the difference is Hilbert-Schmidt.  
\begin{lemma}\label{lem:sim}
Let $g_0,g_1$ be two Riemannian metrics on $\mathsf{H}$, with $g_0\sim g_1$. Then the operator $C=g_1^\sharp\circ g_0^\flat$ has spectrum contained in the set of positive real numbers, and in particular has a well-defined square root $A=C^{1/2}$. We have that $A\sim I$, and 
\[ g_1(Av,Aw)=g_0(v,w)\]
for all $v,w\in \mathsf{H}$. 
\end{lemma}
\begin{proof}

Since $g_1(Cv,w)=g_0(v,w)=g_1(v,Cw)$ for all $v,w\in \mathsf{H}$, the operator $C$ is symmetric with respect to $g_1$. Furthermore, it is non-negative since 
$g_1(Cv,v^*)=g_0(v,v^*)\ge 0$ for all $v\in \mathsf{H}^\C$. To show that $0$ is not in the spectrum, let $\lambda\in \C$, with $\lambda\neq 1$. Since $g_0\sim g_1 \Rightarrow C\sim I$, we have
$C-\lambda I\sim (1-\lambda)I$, an invertible operator. Since Hilbert-Schmidt operators are compact, this shows that $C-\lambda I$ is Fredholm of index zero, and hence is invertible if and only if its kernel is 
zero. But if $v\in \mathsf{H}^\C$ is in the kernel, then $Cv=\lambda v$, hence 
\[ \lambda g_1(v,v^*)=g_1(Cv,v^*)=g_0(v,v^*),\]
thus  $\lambda>0$ or $v=0$. This shows $\on{spec}(C)\subset (0,\infty)$. It follows that 
the square root  $A=C^{1/2}$ is defined. $A$ is again a positive symmetric operator with respect to $g_1$, and since $C\sim I$ we have that $A\sim I$. It satisfies $g_1(Av,Aw)=g_1(Cv,w)=g_0(v,w)$ as required. 
\end{proof}

The  isometry $A\colon (\mathsf{H},g_0)\to (\mathsf{H},g_1)$ for $g_0\sim g_1$ extends to an isomorphism of Clifford algebras 
\[ A\colon \Cl(\mathsf{H},g_0)\to \Cl(\mathsf{H},g_1).\] 
If $J_0$ is a $g_0$-orthogonal complex structure, and $J_1$ is a $g_1$-orthogonal complex structure, then  there exists an isomorphism of spinor modules $\Sz_{\mathsf{H},J_0}\to \Sz_{\mathsf{H},J_1}$ compatible with the isomorphism of Clifford algebras if and only if $J_0\sim J_1$.  Indeed, using $A\sim I$ the condition is equivalent to $J_0\sim A\circ J_1\circ A^{-1}$, hence the 
claim follows from Theorem \ref{th:shasp}. Proposition \ref{prop:subspace} generalizes similarly. 

\subsection{Compatible complex structures in infinite dimensions}\label{subsec:compatiblecs}
Let $\mathsf{H}$ be a separable real Hilbert space, with a (strongly) symplectic 2-form $\omega\colon \mathsf{H}\times \mathsf{H}\to \R$. That is, $\omega$ is non-degenerate in the sense that the associated map $\omega^\flat\colon \mathsf{H}\to \mathsf{H}^*$ is an isomorphism. A complex structure $J$ on $\mathsf{H}$  is \emph{compatible} with $\omega$ if 
\[ g(v,w)=\omega(v,Jw)\]
is a Riemannian metric on $\mathsf{H}$. Such a complex structure $J$ is orthogonal with respect to $g$ and symplectic with respect to 
$\omega$, and the choice of $J$ makes $\mathsf{H}$ into a complex Hilbert space (cf. Section 1.1 in \cite{che:ma}).  We denote by $\ca{J}(\mathsf{H},\omega)$ the set of $\omega$-compatible complex structures.   

\begin{lemma}\label{lem:hsequivalent}
If $J_0,J_1\in\ca{J}(\mathsf{H},\omega)$, with corresponding Riemannian metrics $g_0,g_1$, then $J_0\sim J_1$ if and only if $g_0\sim g_1$. 
\end{lemma}
\begin{proof}
We have $J_0\sim J_1$ if and only if $g_1^\sharp\circ g_0^\flat=J_1^{-1}\circ J_0\sim I$. But 
this means $g_0^\flat\sim g_1^\flat$. 
\end{proof}
Given two $J$'s in $\ca{J}(\mathsf{H},\omega)$, there exists a real-linear automorphism $A$ of  $\mathsf{H}$ intertwining the two resulting Hilbert space inner products. Equivalently, $A$ preserves $\omega$ and intertwines the two $J$'s. This shows that $\ca{J}(\mathsf{H},\omega)$ is a homogeneous space for the group $\on{Sp}(\mathsf{H})$ of symplectic transformations, with stabilizer at $J\in \ca{J}(\mathsf{H},\omega)$ the corresponding unitary group $\on{U}_J(\mathsf{H})$. 

The \emph{restricted symplectic group} $\on{Sp}_{res}(\mathsf{H})$ consists of all $A \in \on{Sp}(\mathsf{H})$ satisfying $AJA^{-1}\sim J$; it is a Banach Lie group for the topology induced by
\[ \|A\|_J=\|A\|+\|[J,A]\|_{HS},\]
where $\|-\|$ is the operator norm, and $\|-\|_{HS}$ is the Hilbert-Schmidt norm with respect to the Riemannian metric defining the Hilbert space structure (\cite{pr:lo}).
By the same argument as above, $\on{Sp}_{res}(\mathsf{H})$, acts transitively on the set $\ca{J}_{res}(\mathsf{H},\omega)$ of compatible complex structures in the given polarization class.  

The spaces $\ca{J}(\mathsf{H},\omega)$ and $\ca{J}_{res}(\mathsf{H},\omega)$ become Banach manifolds, such that the quotient maps
\[ \on{Sp}(\mathsf{H})\rightarrow \on{Sp}(\mathsf{H})/\on{U}_J(\mathsf{H}) \cong \ca{J}(\mathsf{H},\omega), \qquad \on{Sp}_{res}(\mathsf{H})\rightarrow \on{Sp}_{res}(\mathsf{H})/\on{U}_J(\mathsf{H}) \cong \ca{J}_{res}(\mathsf{H},\omega) \]
are submersions.  Similar to finite dimensions, polar decomposition gives a contraction of $\on{Sp}(\mathsf{H})$ (resp. $\on{Sp}_{res}(\mathsf{H})$) onto $\on{U}_J(\mathsf{H})$, showing that $\ca{J}(\mathsf{H},\omega)$ (resp. $\ca{J}_{res}(\mathsf{H},\omega)$) is contractible; we discuss this briefly in Appendix \ref{app:contractible}.  The following gives an explicit path connecting two compatible complex structures defining the same polarization class. 
\begin{theorem}
Let $\omega$ be a symplectic structure on a real Hilbert space $\mathsf{H}$, and let 
$J_0,\ J_1$ be two compatible complex structures, defining the same polarization. 

For all $t\in [0,1]$, the spectrum of $K_t=(1-t)J_0+t J_1$ is contained in the set of non-zero imaginary numbers. The operators 
\[ J_t=K_t\ (-K_t^2)^{-1/2}\] 
are well-defined compatible complex structures connecting $J_1$ and $J_0$ within 
the given polarization class. 
\end{theorem}
\begin{proof}
Suppose $\lambda\in\C $ is in the spectrum of $K_t$, so that  $K_t-\lambda I$ is not invertible. 
Unless $\lambda=\pm \sqrt{-1}$, the operator  $K_t-\lambda I$ 
is a compact perturbation of  an invertible operator $J_0-\lambda I$; hence it is 
non-invertible if and only of its kernel is non-zero. Let  $v\in \mathsf{H}^\C$ be a non-zero element in the kernel. Then 
$\omega(v^*,(K_t-\lambda I)v)=0$, hence  
\[  (1-t)g_0(v^*,v)+t g_1(v^*,v)= \lambda \omega(v^*,v).\]
For $0\le t\le 1$, the left hand side is $>0$, hence so is the right hand side. 
On the other hand, 
since $\omega$ is a real 2-form, $\omega(v^*,v)\in \sqrt{-1}\R$. This shows 
that that the spectrum of $K_t$ lies in the set of non-zero imaginary numbers; hence  $-K_t^2$ has spectrum in strictly positive real numbers. It follows that  $J_t=K_t (-K_t^2)^{-1/2}$ are well-defined complex structures. To show that $J_t\sim J_0$, it suffices 
to show that $(-K_t^2)^{-1/2}\sim I$. But this follows because $R_t:=I+K_t^2\sim 0$, and since 
\[ (-K_t^2)^{-1/2}-I=R_t f(R_t),\]
with $f(x)=\f{1}{x}(\f{1}{\sqrt{1-x}}-1)$, is the product of a Hilbert-Schmidt operator and a bounded operator. 
\end{proof}

The choice of an $\omega$-compatible complex structure on $\mathsf{H}$ determines a spinor module $\Sz_\mathsf{H}$. The discussion above shows that equivalent compatible complex structures $J_0\sim J_1$ are homotopic within their equivalence class, and that they determine a homotopy class of Hilbert-Schmidt equivalent Riemannian metrics. Hence they give isomorphic (in fact, homotopic) spinor modules. 

\subsection{The spin representation of the loop group}\label{subsec:centext}
Let $G$ be a compact, connected Lie group, and $LG$ the space of $G$-valued loops  
of Sobolev class $s$. If $s>\hh$, then $LG$ consists of continuous loops, and is a Banach Lie group under pointwise multiplication,  with Lie algebra $L\g$ the $\g$-valued loops of Sobolev class $s$. We refer to $LG$ for a fixed choice of $s>\hh$ as the loop group of $G$. 
The choice of an invariant inner product on $\g$ defines an $LG$-invariant weak Riemannian metric $g$ on $L\g$, given by the pointwise inner product followed by integration over $S^1$, using the standard volume form $\d  t$ on $S^1$. It extends to a 
(strong) Riemannian metric on the real Hilbert space ${\bf L}\g$ of square integrable loops (i.e., Sobolev class $0$). 
Its complexification has a triangular decomposition 
\[ {\bf L}\g^\C=({\bf L}\g^\C)^+\oplus \g^\C \oplus ({\bf L}\g^\C)^-\]
given as the spans of positive, zero, and negative Fourier modes, respectively. Let $J_0\in \mathbb{B}(\mathbf{L}\g)$ be the operator
whose complexification is equal to $-\sqrt{-1}$ on negative Fourier modes, $0$ on zero Fourier modes, and $+\sqrt{-1}$ on positive Fourier modes. 
\begin{remark}
Equivalently, 
\[ J_0=\partial_0/|\partial_0|,\]
where $\partial_0$ is the unbounded skew-adjoint operator, given as the exterior derivative 
(we use the standard volume form $\partial t$ on $S^1$ to identify 1-forms and functions). 
On the loop $z\mapsto X\otimes z^n$ in $L\g^\C$, for $X\in \g^\C$ and $n\in \Z$, it is given by 
$\partial_0(X \otimes z^n)=2\pi i n X\otimes z^n$ . The kernel of $\partial_0$ are the constant loops, and $|\partial_0|^{-1}$ denotes the inverse of $|\partial_0|$ on $\g^\perp\subset {\bf L\g}$, extended by zero. 
\end{remark}
In order to have an actual complex structure, we consider the larger space ${\bf L}\g\oplus \g$, with the complex structure 
$J$ given as the sum of $J_0$ with the standard complex structure $(X,Y)\mapsto (-Y,X)$ on $\ker(J_0)\oplus \g=\g\oplus \g$.
 Then $J$ determines a $\Z_2$-graded spinor module 
\[ \Cl( {\bf L}\g\oplus \g )\circlearrowright \Sz_{{\bf L}\g\oplus \g}.\]
The operator $J_0$ is not $LG$-invariant, but one can show (see e.g. \cite{pr:lo}) that the polarization class is preserved. 
Hence, the adjoint action defines a group homomorphism from $LG$ into 
the restricted orthogonal group, and so $LG$ inherits a central extension from that of 
$\on{O}_{res}({\bf L}\g\oplus \g)$. We let 
\begin{equation}\label{eq:centext}
1\to \U(1)\to \wh{LG}^{\on{spin}}\to LG\to 1
\end{equation}
be the \emph{opposite} of this central extension of $LG$ defined in this way; 
it will be referred to as the \emph{spin}-central extension of the loop group. 
We refer to the action on the dual spinor module
$\Hom(\Sz_{{\bf L}\g\oplus \g},\C)$ as the \emph{spin representation} of 
the loop group. By definition, the central circle in \eqref{eq:centext} acts with weight one. 
The spin representation of the loop group (or of its Lie algebra) was described in Kac-Peterson \cite{kac:sp}, see also Araki \cite{ara:bo}. 

\begin{remark}
For $G$ simple and simply connected, the isomorphism classes of central extensions of $LG$ are indexed  by their \emph{level} $k\in\Z$, 
corresponding to the $k$-th power of the \emph{basic} central extension $\wh{LG}$. The level of the spin central extension is the dual Coxeter number $\cox$ of $G$. As a $\wh{LG}$-representation, it is a direct sum of irreducible representations of  highest weight $(\rho,\cox)$. See Appendix \ref{app:central}.
\end{remark}
\begin{remark}
One can also consider the action of larger group $LG\times G$ on ${\bf L}\g\oplus \g$; the central extension of $LG$ extends to the product. If $G$ is simply connected, then the 
central extension is trivial over $G$ (with a unique trivialization), hence $\Sz_{\bf{L}\g\oplus \g}$ has an action of $\wh{LG}^{\on{spin}}\times G$. 
\end{remark}

\section{Hamiltonian loop group spaces and quasi-Hamiltonian spaces}\label{sec:ham}
In this section we will recall the basic definitions for Hamiltonian loop group spaces 
and quasi-Hamiltonian $G$-spaces. Improving on the discussion in \cite{al:mom}, we will phrase the 1-1 correspondence between such spaces as a Morita equivalence (see \cite{xu:mom} for related ideas).

\subsection{The space of connections}
Let $G$ be a compact, connected Lie group, with Lie algebra $\g$. For fixed $s\in\R$,  denote by $LG$ the Banach manifold 
of maps $S^1\to G$ of Sobolev class $s$. We will assume $s>\f{1}{2}$, so that $LG$ consists of continuous loops, and is a Banach Lie group under pointwise multiplication. 
Let $\A$ be the space of connections, given as $\g$-valued 1-forms on $S^1$ of Sobolev class $s-1$. 
The loop group $LG$ acts smoothly on the Banach manifold $\A$ by gauge transformations:
\[ \lambda\cdot \mu=\Ad_\lambda(\mu)-\partial \lambda\  \lambda^{-1},\ \ \ \ \ \mu\in\A,\  \lambda\in LG.\]
Here we are reserving the notation $\partial$ for the exterior derivative on $S^1$; thus 
$\partial \lambda\  \lambda^{-1}$ is the pull-back of the right-invariant Maurer-Cartan form
$\theta^R\in \Omega^1(G,\g)$ under the map $\lambda\colon S^1\to G$. 
The generating vector fields $\xi_\A,\ \xi\in L\g$ for the gauge action are the covariant derivatives: At  $\mu\in \A$, we have that $\xi_\A|_\mu=\partial_\mu\xi$ where  
\[ \partial_\mu=\partial+\ad_\mu \colon L\g \to T_\mu \A.\]
\begin{remark}
To avoid the dependence on a Sobolev level, one could also work with Frechet manifolds, and consider smooth loops and connections. See \cite{nee:bo} for foundational results regarding the smooth loop group. 
\end{remark}

\subsection{The path fibration}\label{subsec:pathfib}
Denote by $\ca{P}G$ the Banach manifold of all paths $\gamma\colon \R\to G$ of Sobolev class $s$, with the property that $\gamma(t+1)\gamma(t)^{-1}$ is constant. 
The direct product $LG\times G$ acts on $\ca{P}G$ by 
\[ ((\lambda,g).\gamma)(t)=g\,\gamma(t)\,\lambda(t)^{-1}.\]
The $LG$-action makes $\ca{P}G$ into a $G$-equivariant principal $LG$-bundle over $G$, with quotient map $q(\gamma)=\gamma(1)\gamma(0)^{-1}$, where the $G$-action on the base $G$ is given by conjugation. This is called the \emph{path fibration}. On the other hand, the $G$-action makes  $\ca{P}G$ into an $LG$-equivariant principal $G$-bundle over the space of connections, with quotient map $p\colon \ca{P}G\to \A, \gamma\mapsto \gamma^{-1}\partial \gamma$. 
We arrive at the correspondence diagram 
\begin{equation}\label{eq:corrpg}
 \xymatrix{ & \P G\ar[dl]_p\ar[dr]^q & \\
\mathcal{A} & & G
}\end{equation}
\begin{remarks}\label{rem:4remarks}
\begin{enumerate}
\item
The bundle $p\colon \P G\to \A$ has a distinguished section, with image the space $\P_eG$ of based paths (i.e., paths with $\gamma(0)=e$).  Note however that since this trivialization of the bundle $p\colon \P G\to A$  involves an evaluation map, it does not have good properties for the Sobolev $\hh$ completions to be considered later. 
\item
The holonomy map $\Hol\colon \A\to G$, taking the parallel transport of a connection around $S^1$, 
may be defined as the inclusion 
$\A\to \P G$ followed by $q$. Note however that $\Hol\circ\, p\not=q$. Indeed, 
for a connection $\mu=p(\gamma)$ we have that $\Hol(\mu)=\gamma(0)^{-1}\gamma(1)$ whereas $q(\gamma)=\gamma(1)\gamma(0)^{-1}$. 
\item The pull-back of the bundle $q\colon \P G\to G$ under the exponential map $\exp\colon\g\to G$ is canonically trivial; 
the trivializing section takes $X\in\g$ to the  path $t\mapsto \exp(tX)$. 
\end{enumerate}
\end{remarks}
%
%
%
The correspondence diagram \eqref{eq:corrpg} may be seen as a Morita equivalence of groupoids,  between the action  groupoids $[G/G]$ for the conjugation action and $[\A/LG]$ for the gauge action, with $\P G$ as the equivalence bimodule. In particular, the stabilizer groups for the actions, as well as the `transverse geometry' to orbits, are isomorphic. To be precise, let  $\gamma\in \P G$, mapping to the elements 
$\mu=p(\gamma)\in \A$ and $a=q(\gamma)\in G$. The projections 
from $LG\times G$ to the two factors induce isomorphisms of the stabilizer groups,
\begin{equation}\label{eq:isomorphicgroups}
 (LG)_\mu \longleftarrow (LG\times G)_\gamma \longrightarrow G_a,
 \end{equation}
The resulting isomorphism $G_a\cong LG_\mu$ takes $h\in G_a$ to the loop 
$\lambda(t)=\Ad_{\gamma(t)^{-1}}h$. 
Furthermore, there are slices 
\begin{align*}
V_\gamma\subset \P G &\ \ \  \mbox{ for the $LG\times G$-action at $\gamma$}\\
V_\mu\subset \A &\ \ \ \mbox{  for the $LG$-action at $\mu$}\\
V_a\subset G &\ \ \ \mbox{ for the $G$-action  at $a$}
\end{align*}
invariant under the respective groups \eqref{eq:isomorphicgroups}, 
in such way that $p$ and $q$ restrict to diffeomorphisms 
\begin{equation}\label{eq:isomorphicslices}
 V_\mu \longleftarrow V_\gamma \longrightarrow V_a,
 \end{equation}
equivariant with respect to \eqref{eq:isomorphicgroups}. 
%
Letting $U_a\subset \g_a$ be a sufficiently small $\Ad(G_a)$-invariant open neighborhood of the origin, we may take 
\[ V_\mu=\{\mu+\Ad_{\gamma^{-1}}(X\,\partial t)|\ X\in U_a\},\ \ \ 
V_\gamma=\{\exp(tX)\gamma(t)|\ X\in U_a\},\ \ \ 
V_a=\{\exp(X)a|\ X\in U_a\}.\]
The slice $V_\mu$ can also be written as 
\[ V_\mu
=\{\mu+\zeta\partial t|\ \zeta\in U_\mu\} \]
where $U_\mu\subset L\g_\mu$ is the image of $U_a$ under the isomorphism $L\g_\mu\cong \g_a$. We will often use these slices in conjunction with partitions of unity: Since $G$ is compact, it is covered by finitely many 
of the flow-outs $G. V_a$ of slices, and we may choose an associated 
$G$-invariant partition of unity. Using the diagram \eqref{eq:corrpg}, this determines finite open covers of $\P G$ and of $\A$, along with invariant partitions of unity.  

Suppose now that the Lie algebra $\g$ comes with an $\Ad$-invariant metric $\cdot$.  Let $\theta^L,\theta^R\in\Omega^1(G,\g)$ denote the Maurer-Cartan-forms. The metric on $\g$ determines a bi-invariant \emph{Cartan 3-form} $\eta=\f{1}{12}\theta^L\cdot[\theta^L,\theta^L]\in \Omega^3(G)$, and its $G$-equivariant extension $\eta_G(X)=\eta-\hh(\theta^L+\theta^R)\cdot X$ for $X\in\g$.  The 3-form $\eta$ is closed, and $\eta_G$ is equivariantly closed: 
\[ (\d-\iota(X_G))\eta_G(X)=0.\]
The metric also determines a certain $LG\times G$-invariant 2-form $\varpi\in \Omega^2(\P G)$, with the properties 
\[ \d\varpi=-q^*\eta,\ \ \ \ \iota(\xi_{\P G})\varpi=p^*\d\l\mu,\xi\r,\ \ \ \ \ \iota(X_{\P G})\varpi=-q^*\Big(\hh(\theta^L+\theta^R)\cdot X\Big)
\]
for $\xi\in L\g$ and $X\in\g$. Letting $\ev_t\colon \PG\to G$ be the evaluation map, this 2-form is given by (cf.~ Appendix \ref{app:2form})
\[ \varpi=\hh \int_0^1 \Big(\on{ev}_t^*\theta^R\cdot \f{\p}{\p t}  \on{ev}_t^*\theta^R\Big)\, \partial t\ 
+ \hh \on{ev}_0^*\theta^L\cdot \on{ev}_1^*\theta^L.
\]
See \cite{al:ati,voz:lo} for a conceptual construction of this 2-form. 
\begin{remark}
These formulas have a  Dirac-geometric significance: The group $G$ carries the so-called 
\emph{Cartan-Dirac structure}, and the total space of the action groupoid $[G/G]$ is a quasi-presymplectic groupoid integrating this Dirac structure (see \cite{bur:int,xu:mom}) . On the other hand, $\A$ carries an infinite-dimensional `Lie-Poisson' Dirac structure \cite{cab:dir}, and  $[\A/LG]$, with a certain 2-form, is its integration.  The correspondence diagram \eqref{eq:corrpg} for the path fibration $\P G$, together with the 2-form $\varpi$, makes it a \emph{dual pair} of Dirac manifolds, in the sense of \cite{fre:rea}, and also defines a Morita equivalence of pre-symplectic groupoids, see \cite[Proposition 4.26]{xu:mom}.
\end{remark}

\subsection{Hamiltonian loop group spaces}
The  invariant inner product $\cdot$ on $\g$ defines a pairing between 
$\A$ and $L\g$, given by pointwise inner product followed by integration over $S^1$. 
\begin{definition}\cite{me:lo}\label{def:melo}
A \emph{Hamiltonian $LG$-space} $(\M,\omega,\Phi)$ is a Banach manifold $\M$, equipped with a smooth action of $LG$, an invariant weakly symplectic 2-form $\omega$, and a smooth $LG$-equivariant map $\Phi\colon \M\to \A$ satisfying the \emph{moment map condition}
\[ \iota(\xi_\M)\omega=-\l\d\Phi,\xi\r,\ \ \ \xi\in L\g.\]
It is called \emph{proper} if the moment map $\Phi$ is proper.
\end{definition}
The first examples of Hamiltonian $LG$-spaces are the coadjoint orbits $\O=LG\cdot\mu\subset \A$, with the 2-form given on generating vector fields by 
\[ \omega((\xi_1)_\O,\ (\xi_2)_\O)|_\mu=\int_{S^1} \partial_\mu \xi_1\cdot\xi_2,\ \ \ \ \  \xi_1,\xi_2\in L\g;\]
the moment map is the inclusion. Another natural example is the moduli space of flat connections on surfaces with boundary. There are many other examples; some of these are best understood from the correspondence with quasi-Hamiltonian spaces, which we discuss next.

\subsection{Quasi-Hamiltonian spaces}

\begin{definition}\label{def:quas}
A \emph{quasi-Hamiltonian $G$-space} $(M,\omega_M,\Phi_M)$ is a (finite-dimensional) $G$-manifold $M$ with an invariant 2-form $\omega_M$ and a $G$-equivariant smooth map 
$\Phi_M\colon M\to G$ satisfying the following conditions: 
\begin{enumerate}
\item $\d\omega_M=-\Phi_M^*\eta$, 
\item $\iota(X_M)\omega_M=-\Phi_M^*(\hh(\theta^L+\theta^R)\cdot X)$ for all $X\in\g$ 
\item $\ker(\omega_M)\cap \ker(T_M\Phi)=0$. 
\end{enumerate}
\end{definition}
The first two conditions may be combined into a single requirement 
\[ (\d-\iota(X_M))\omega_M=-\Phi_M^*\eta_G(X),\]
stating that $\omega_M$ is an equivariant primitive of $-\Phi_M^*\eta_G$. 

Basic examples are the conjugacy classes $\Co \subset G$, with moment map the inclusion, and the double $D(G)=G\times G$, with moment map the group commutator. Other examples include $\SU(n)$ acting on an even-dimensional sphere $S^{2n}$ \cite{hu:imp}, and $\on{Sp}(n)$  acting on quaternionic projective spaces \cite{esh:ex} or on  
quaternionic Grassmannians \cite{kno:mul}. A full classification of  such \emph{multiplicity-free} quasi-Hamiltonian spaces was obtained in the work of Knop \cite{kno:mul}, with many new examples.  

There is a correspondence between proper Hamiltonian $LG$-spaces $\M$ and quasi-Hamiltonian $G$-spaces $M$, described by a diagram 
\begin{equation}\label{eq:corrn} 
\xymatrix{ & \ca{N}\ar[dl]_p\ar[dr]^q & \\
\M & & M
}\end{equation}
Given $\M$, the space $\N$ is the pullback of  the bundle $p\colon \P G\to \A$ under the moment map $\Phi$, and $M=\N/LG$. Conversely, given $M$, the space  $\N$ is the pullback of  the bundle $q\colon \P G\to G$ under the moment map $\Phi_M$, and $\M=\N/G$. 
The 2-forms are related by 
\[ q^*\omega_M-p^*\omega=\Phi_{\N}^*\varpi.\]
The correspondence may be used to define examples of Hamiltonian loop group spaces by construction of the corresponding quasi-Hamiltonian spaces.  For a coadjoint orbit $\O\subset \A$, the associated quasi-Hamiltonian spaces is a conjugacy class $\Co\subset G$. 
\

\section{The spinor bundle for a Hamiltonian $LG$-space}\label{sec:spin}
Our construction of a spinor bundle starts out by showing that the weakly symplectic 2-form of any proper Hamiltonian loop group space becomes \emph{strongly} symplectic on a suitable  completion of the tangent bundle. The construction is particularly nice for the case of coadjoint orbits; hence we will discuss this case first.

\subsection{The spinor bundle over coadjoint loop group orbits}
Let $\O\subset \A$ be a coadjoint orbit for the loop group. The weak symplectic form on $\O$ is 
given on tangent spaces $T_\mu\O=L\g/L\g_\mu$ by 
 \begin{equation} \label{eq:coadsympl}
 \omega((\xi_1)_\O,(\xi_2)_\O)|_\mu=\int_{S^1} \partial_\mu\xi_1\cdot \xi_2,\ \ \xi_i\in L\g.\end{equation}
The topological dual space of $L\g$ (consisting of $\g$-valued loops of Sobolev class $s$) are  $\g$-valued 1-forms of Sobolev class $-s$, hence the dual space of $T_\mu\O=L\g/L\g_\mu$ is the annihilator 
of $L\g_\mu$ inside the space of 1-forms of Sobolev class $-s$. But 
\[ (\omega_\mu)^\flat\colon T_\mu\O\to (T_\mu\O)^*,\ \ \xi_\O|_\mu \mapsto \partial_\mu\xi,\] 
takes values in the subspace of 1-forms of Sobolev class $s-1$. Since $s>\f{1}{2}$, we have that $s-1
>-s$. This verifies that $(\omega_\mu)^\flat$ is \emph{not} onto. On the other hand, letting 
\[ \ol{L}\g:=L_{(\f{1}{2})}\g\]
be the loops of Sobolev class $\f{1}{2}$, we see that $\omega_\mu$ extends to a \emph{strongly} 
symplectic bilinear form on $\ol{T}_\mu\O=\ol{L}\g/L\g_\mu=\ol{L}\g_\mu^\perp$:
\begin{equation}\label{eq:ommu}
\omega_\mu\colon \ol{L}\g_\mu^\perp\times \ol{L}\g_\mu^\perp\to \R
\end{equation}
Here  $\perp$ stands for the orthogonal space relative to the (weak) metric on $\ol{L}\g$. Hence, the bundle $\ol{T}\O$ becomes a strongly symplectic vector bundle over $\O$. There is a canonical $LG$-invariant $\omega$-compatible complex structure on $\ol{T}\O$, given at $\mu$ by the $LG_\mu$-invariant complex structure
\begin{equation}\label{eq:jmu} 
J_\mu=\partial_\mu/|\partial_\mu| \in \mathbb{B}(\ol{L}\g_\mu^\perp) \end{equation}
The resulting Riemannian metric on $\ol{T}\O$ defines a Clifford bundle 
$\Cl(\ol{T}\O)$, and the complex structure gives an $LG$-equivariant spinor module 
\[ \Cl(\ol{T}\O)\circlearrowright \Sz_{\ol{T}\O}.\]
\begin{remark}
The bracket on $L\g$ does not extend continuously to $\ol{L}\g$; hence $\ol{L}\g$ 
does not become a Banach Lie algebra. Similarly the $G$-valued loops of Sobolev class $\f{1}{2}$ are not a Banach Lie group.
\end{remark}

\subsection{The spinor bundle for proper Hamiltonian $LG$-spaces}
We will now extend this construction to arbitrary proper Hamiltonian $LG$-spaces. For any Banach manifold $Q$ with a proper $LG$-action, we let
\[ \ol{T}Q:=(TQ\times \ol{L}\g)/\sim\]
be the quotient space under the relation 
\[ (v+\xi_Q|_q,0)\sim (v,\xi),\ \ \ v\in T_qQ,\ \ \xi\in L\g\subset \ol{L}\g,\] 
where $\xi_Q$ denotes the vector field on $Q$ generated by $\xi\in Lg$. 
The properness assumption implies that the action has finite-dimensional slices with compact
stabilizer groups $LG_q$, and the definition amounts to replacing the orbit 
directions $T_q(LG\cdot q)=L\g/(L\g)_q$ with $\ol{L}\g/(L\g)_q$. In particular, $\ol{T}\A$ is defined, and for every proper Hamiltonian $LG$-space $\M$ the bundle 
$\ol{T}\M$ is defined
(since properness of the moment map implies properness of the action) . 
Equivariant maps of Banach manifolds $Q_1\to Q_2$ with proper $LG$-actions  define equivariant bundle maps $\ol{T} Q_1\to \ol{T} Q_2$.

\begin{theorem}\label{th:completedtangentbundle}
Let $(\M,\omega,\Phi)$ be a proper Hamiltonian $LG$-space. Then 
\begin{enumerate}
\item the 2-form $\omega$ extends to a strongly symplectic form on $\ol{T}\M$,
\item the bundle $\ol{T}\M$ has a distinguished $LG$-invariant polarization, defining 
$\ca{J}_{\on{res}}(\ol{T}\M,\omega)$, 
\item there exists a global $LG$-invariant compatible complex structure $J\in \ca{J}_{\on{res}}(\ol{T}\M,\omega)$, within this polarization class. 
\end{enumerate}
\end{theorem}
\begin{proof}
For every $m\in \M$, with image $\mu=\Phi(m)$, we have a subspace of finite codimension
\[ \ol{L}\g_\mu^\perp\subset \ol{T}_m \M\] 
embedded via the action map. On this subspace, we have the standard complex structure $J_\mu$ defined in \eqref{eq:jmu}. This determines a polarization $\ca{J}_{\on{res}}(\ol{T}_m\M)$, consisting of those complex structures on $\ol{T}_m\M$ that differ from $J_\mu$ by a Hilbert-Schmidt operator.

We will use the \emph{symplectic cross-section theorem} for Hamiltonian loop group actions
\cite[Section 4.2]{me:lo}.  Given $m\in \M$ with image $\mu=\Phi(m)$, let $V_\mu\subset \A$ be the  slice through $\mu$ as described in Section \ref{subsec:pathfib}. Its pre-image $Y=\Phi^{-1}(V_\mu)$ is a finite-dimensional  $LG_\mu$-invariant symplectic submanifold of $\M$. We have the $LG_\mu$-equivariant $\omega$-orthogonal decomposition
\begin{equation}\label{eq:tmdecomp}
T\M|_Y=TY\oplus (Y\times L\g_\mu^\perp),\end{equation}
where the second summand is embedded by the generating vector fields. The weak symplectic structure on $\{y\}\times L\g_\mu^\perp$ for $y\in Y$ is determined by the moment map condition, and is given by 
\[
 \omega((\xi_1)_\M,(\xi_2)_\M)|_y=\int_{S^1} \partial_{\nu}\xi_1\cdot \xi_2,\ \ \xi_i\in L\g_\mu^\perp
\]
where $\nu=\Phi(y)$ (cf.~ \eqref{eq:coadsympl},\eqref{eq:ommu}). This 2-form extends to a strongly symplectic form on 
$\{y\}\times \ol{L}\g_\mu^\perp$, and so $\omega|_Y$ extends to an $LG_\mu$-invariant strongly symplectic 2-form 
on 
\begin{equation}\label{eq:tmdecomp1}
\ol{T}\M|_Y\cong TY\oplus (Y\times \ol{L}\g_\mu^\perp),\end{equation}
and by equivariance to an $LG$-invariant symplectic form on $\ol{T}\M$.

The summand $Y\times \ol{L}\g_\mu^\perp$ in \eqref{eq:tmdecomp1} has an $LG_\mu$-invariant compatible complex structure, 
given by \eqref{eq:jmu}. By choosing an $LG_\mu$-invariant
compatible complex structure on $TY$, we obtain an $LG_\mu$-invariant compatible complex structure on $\ol{T}\M|_Y$, hence an $LG$-invariant compatible complex structure $J_{LG.Y}$ on $\ol{T}\M|_{LG.  Y}$. Given another symplectic cross-section $Y'$, the difference between $J_{LG.Y},\ J_{LG.Y'}$ over $LG.Y\cap LG.Y'$ has finite rank; hence the corresponding Riemannian metrics agree on a subbundle 
of finite codimension. Using a (finite) partition of unity as in Section \ref{subsec:pathfib}, these local 
definitions  may be patched together, and we obtain an $LG$-invariant compatible Riemannian metric and associated complex structure $J$ on all of $\ol{T}\M$. By construction, this complex structure $J$ 
on $\ol{T}_m\M$, at any given $m\in \M$, with image $\mu=\Phi(m)$, differs from the complex structure $J_\mu$ on $\ol{L}\g_\mu^\perp\subset \ol{T}_m\M$  by a finite rank operator. In particular, it is independent of the choices modulo Hilbert-Schmidt equivalence. 
\end{proof}

Let $g$ be the Riemannian metric on $\ol{T}\M$ defined by $\omega$ and $J$. We obtain a Clifford bundle with an $LG$-equivariant spinor module, 
\begin{equation}\label{eq:thespinormodule}
 \Cl(\ol{T}\M)\circlearrowright \Sz_{\ol{T}\M}.\end{equation}
As explained in Section \ref{subsec:compatiblecs}, two choices $J_0, J_1$ of compatible complex structures in the given polarization class give rise to equivalent Riemannian metrics $g_0\sim g_1$, hence the Clifford algebras are canonically identified. A homotopy between $J_0,\,J_1$ within the polarization class
gives a homotopy of spinor modules.

\begin{remark}
On a (finite-dimensional) K\"ahler manifold,  one takes the spinor module for the Dolbeault operator as $\Sz_{TM}:=\wedge (T^*M)^{0,1}$. This is consistent with our conventions, since the isomorphism $g^\flat \colon TM\cong T^*M$ given by the metric 
identifies $(TM)^+=(TM)^{(1,0)}\cong (T^*M)^{(0,1)}$. 
\end{remark}

\section{The twisted spin-c -structure for a quasi-Hamiltonian $G$-space}\label{sec:twist}

Having obtained the spinor module $\Sz_{\ol{T}\M}$ over the Hamiltonian $LG$-space $\M$, we will use it to construct the twisted $\Spin_c$-structure for the associated quasi-Hamiltonian $G$-space $M$. Our strategy is to use the correspondence diagram \eqref{eq:corrn} to `pull back' the spinor bundle over $\M$ under the projection $p\colon \N\to \M$, and then `descend' to $M$ under the map $q\colon \N\to M$. This will require the use of principal connections, obtained as pullbacks of principal connections from the path fibration. As we will see, the appearance of a central extension of the loop group will prevent us from actually descending the spinor module to $M$; this is the reason why we will not obtain an actual $\Spin_c$-structure on $M$, in general, but only a twisted $\Spin_c$-structure.  We will begin our discussion with the case that $\M$ is a coadjoint orbit, with $M$ the associated conjugacy class.

\subsection{The twisted $\Spin_c$-structure for conjugacy classes}
\label{subsec:orbits}
Let $\N\subset \P G$ be an orbit for the $LG\times G$-action, and $\O=p(\N),\ \Co=q(\N)$ the corresponding coadjoint $LG$-orbit and $G$-conjugacy class:
\begin{equation}\label{eq:oorrn} 
\xymatrix{ & \ca{N}\ar[dl]_p\ar[dr]^q & \\
\O & & \Co
}\end{equation}
\begin{proposition}[Connections over orbits] 
\label{prop:orbits}
The bundles
$p\colon \N\to  \O,\ q\colon \N\to \Co$ have distinguished invariant connections 
\[ \alpha\in \Omega^1(\N,\g)^{LG\times G},\ \ \beta\in \Omega^1(\N,\g)^{LG\times G},\]
with continuous extensions to the completions, $\ol{\alpha}\colon \ol{T}\N\to \g,\ \ 
\ol{\beta}\colon \ol{T}\N\to \ol{L}\g$.  
\end{proposition}
\begin{proof}
Let $\gamma\in \N$ be given, with $\mu=p(\gamma)\in \O$ and $a=q(\gamma)\in \Co$. 
Let 
$R_\mu\colon L\g\to \g$  be defined by the composition 
\[ R_\mu\colon L\g\to  L\g_\mu \cong \g_a\hra \g\]
where the first map is the orthogonal projection. Note that this 
extends to  the completion, $\ol{R}_\mu\colon \ol{L}\g\to \g$. 
The desired $LG$-invariant principal connection $\alpha$ on $p\colon \N\to \O$ is given by the $(LG\times G)_\gamma$-equivariant map 
\[ \alpha_\gamma\colon  T_\gamma\N\equiv  (L\g\oplus\g)/(L\g\oplus \g)_\gamma\to \g,
\ \ [(\xi,X)]\mapsto X-R_\mu(\xi).\]
This is well-defined, since $X-R_\mu(\xi)=0$ for $(\xi,X)\in (L\g\oplus \g)_\gamma$, and  it 
extends to a map on the completion (given by the same formula, with $R$ replaced by $\ol{R}$). Similarly, let 
 $S_\mu\colon \g\to L\g$ be the composition 
\[ S_\mu\colon \g\to \g_a\cong L\g_\mu\hra L\g\]
where the first map is orthogonal projection. The desired $G$-invariant principal connection $\beta$ on $q\colon \N\to \Co$ is given by the $(LG\times G)_\gamma$-equivariant map
\[ \beta_\gamma\colon T_\gamma\N\equiv   (L\g\oplus\g)/(L\g\oplus \g)_\gamma\to L\g,\ \ \ 
 [(\xi,X)]\mapsto \xi-S_\mu(X). \]
 Its extension to the completion is given by the same formula. 
\end{proof}
For the following result, we will assume that $G$ is simply connected. Then all 
central extensions of $G$ by $\U(1)$ are trivial, and since  $\on{Hom}(G,\U(1))=\{1\}$ the trivialization is in fact unique. 

\begin{theorem}[Twisted Spin-c structure of conjugacy classes]
For every conjugacy class $\Co\subset G$ of a compact, simply connected Lie group $G$, there is  a distinguished spinor module 
\[ \Cl(q^* T\Co)\circlearrowright \Sz_{q^*T \Co},\]
equivariant under $\wh{LG}^{\on{spin}}\times G$.  Here $q\colon \N\to \Co$ is the principal $LG$-bundle $\N\subset \P G$ consisting of all quasi-periodic paths with holonomy in $\Co$. 
\end{theorem}
\begin{proof}
The connections $\alpha,\beta$ constructed above give  $LG\times G$-equivariant isomorphisms of $\ol{T}\N$ with 
$p^*\ol{T}\O\times \g$ and $q^*T\Co \times \ol{L}\g$, respectively. 
Adding another copy of the trivial bundle $\N\times \g$ (with the trivial action of $G$ on $\g$) we obtain $LG\times G$-equivariant isomorphisms
\[  p^*\ol{T}\O\times (\g\oplus\g)\ \cong\ \ol{T}\N\times \g 
\cong\  q^*T\Co \times (\ol{L}\g\oplus \g).\]
The $LG$-invariant symplectic structure, compatible complex structure, and associated Riemannian metric on 
$\ol{T}\O$ pull back to $p^*\ol{T}\O$; together with the standard complex structure and given Riemannian metric on $\g\oplus \g$, these define a  Riemannian metric and orthogonal complex structure on $\ol{T}\N\times\g$. 
The corresponding $LG\times G$-equivariant spinor module $\Sz_{\ol{T}\N\times \g}$ is simply the tensor product of $p^*\Sz_{\ol{T}\O}$ with $\N\times  \wedge \g^\C$. 
(The $G$-action does not preserve the complex structure, but the central extension of $G$ acting on the spinor module is uniquely trivialized.) 

The (completed) tangent space to the fibers of $q\colon \N\to \Co$ is a trivial bundle 
$\ker \ol{T}q\cong \ol{L}\g$, and inherits a metric as a subbundle of $\ol{T}\N$. We will  need an isometric isomorphism of $\ker\ol{T}q=\N\times \ol{L}\g$ with $\N\times \mathbf{L}\g$, where $  \mathbf{L}\g$ has the $L^2$-metric. At $\gamma\in\N$, the  splitting $\ol{T}_\gamma\N\cong \ol{T}_\mu \O\oplus \g$ given by $\alpha$ restricts to the isomorphism 
\[ \ol{L}\g=\ker \ol{T}_\gamma q\to \ol{T}_\mu \O\oplus \g_a,\ \ 
\xi\mapsto (\xi_\O(\mu),\,-R_\mu(\xi)).\]
The resulting metric on $\ol{L}\g $ is 
the direct sum of the metric on $\ol{T}_\mu \O\cong \ol{L}\g_\mu^\perp$ and the 
given metric on $L\g_\mu\cong \g_a\subset \g$. Put differently, it is given by 
\begin{equation} \label{eq:dmu}
(\xi_1,\xi_2)\mapsto \int_{S^1} D_\mu \xi_1\cdot \xi_2
\end{equation}
where $D_\mu$ is the first order pseudo-differential operator given by 
$|\partial_\mu|$ on $\ol{L}\g_\mu^\perp$ and by the identity on $L\g_\mu$. 
(We use the standard volume form on $S^1$ to identify 1-forms with functions.) 
The square root of $D_\mu$ gives an $LG_\mu$-equivariant isomorphism 
\[ D_\mu^{1/2} \colon \ol{L}\g\to {\bf L\g}\]
intertwining the two metrics, and the collection of these operators gives the desired 
$LG\times G$-isometric bundle isomorphism $\ker\ol{T}q\to \N\times  {\bf L\g}$. 
We  may thus define a spinor module over $\Cl(q^*T M)$ as
\begin{equation}\label{eq:conjspinor}
 \Sz_{q^* T\Co}:=\Hom_{\Cl({\bf L\g}\oplus \g)}(\Sz_{{\bf L\g}\oplus \g},
\Sz_{\ol{T}\N\times\g}). 
\end{equation}
It inherits an action of  $\wh{LG}^{\on{spin}}\times G$. 
\end{proof}

The presence of the central extension prevents us from taking a quotient by $LG$ to obtain a spinor module over $\Co$ itself; in this sense we think of $\Sz_{q^* T\Co}$ as a \emph{twisted} $\Spin_c$-structure on $\Co$. There are examples of conjugacy classes 
of simply connected compact Lie groups 
not admitting ordinary $\Spin_c$-structures (let alone canonical ones). See \cite[Example 4.6]{me:conj} and \cite[Appendix D]{lan:cha} for further discussion.

\subsection{Connections on $\P G$}\label{subsec:conn}
To extend this discussion to more general Hamiltonian loop group spaces, we will 
need an $LG$-invariant connection 
$\alpha\in \Omega^1(\P G,\g)$
on the principal $G$-bundle  $p\colon \P G\to \A$, as well as a $G$-invariant connection 
$\beta\in \Omega^1(\P G,L\g)$
on the principal $LG$-bundle $\P G\to G$, with the additional property that the bundle maps
\[ \alpha \colon T \P G\to \g,\ \ \ \beta\colon T\P G\to L\g\] 
extend to the completion $\ol{T}\P G$. 
This requires 
some care: As pointed out in Remark \ref{rem:4remarks}, the principal bundle $p\colon \PG\to \A$ has a canonical trivialization $\P G\to G,\ \ \gamma\mapsto \gamma(0)$, but the corresponding `trivial connection' $\alpha$ does \emph{not} extend to the completion. For example, at the trivial path $\gamma=e$, this connection is $\alpha_e\colon T_e\P G\to \g,\ \xi\mapsto \xi(0)$, which does not extend.

\begin{proposition}[Connections on the path fibration]\label{prop:splittings}
The principal bundles $p\colon \P G\to \A$ and $q\colon \P G\to G$ have invariant principal connections 
\[ \alpha\in \Omega^1(\P G,\g)^{LG\times G},\ \ \ 
\beta\in  \Omega^1(\P G,\,L\g)^{LG\times G},\]
with continuous extensions to the completions, $\ol{\alpha}\colon \ol{T}\P G\to \g,\ \ 
\ol{\beta}\colon \ol{T}\P G\to \ol{L}\g$.  
\end{proposition}
\begin{proof}
We first define a connection over an $LG\times G$-invariant open neighborhood of a given $\gamma\in \P G$. Let $\mu=p(\gamma),\ a=q(\gamma)$, and let  
$V_\gamma\subset \PG,\ V_\mu\subset \A,\ V_a\subset G$ be the slices described in Section \ref{subsec:pathfib}.  Then 
\[ T\P G|_{V_\gamma}=(TV_\gamma\times (L\g\oplus \g))/\sim\]
where the quotient is by the anti-diagonal inclusion of $(L\g\oplus \g)_\gamma$. The completion $\ol{T}\P G|_{V_\gamma}$ has a similar description, with $\ol{L}\g$ on the right hand side. Let $R_\mu,S_\mu$ be as in the proof of Proposition \ref{prop:orbits}. The $(LG\times G)_\gamma$-equivariant map
\begin{equation}\label{eq:spiltp1}
 T\P G|_{V_\gamma}\to\g,\ \ \ [(v,\xi,X)]\mapsto X-R_\mu(\xi),\end{equation}
for $v\in TV_\gamma,\ \xi\in L\g,\ X\in\g$, 
extends to a map on the completions, given by the same formula. It extends by equivariance to an $LG$-invariant connection for $p\colon \P G\to \A$ over 
$LG.V_\gamma\subset \P G$. 

Similarly, the connection on 
$q\colon \P G\to G$ over $LG.V_\gamma$ is given by 
the $(LG\times G)_\gamma$-equivariant map
\begin{equation}\label{eq:splitq1}
 T\P G|_{V_\gamma}\to L\g,\ \ [(v,\xi,X)]\mapsto \xi-S_\mu(X).\end{equation}
It is well-defined, because $\xi-S_\mu(X)=0$ for $(\xi,X)\in (L\g\oplus \g)_\gamma$. 
Having thus constructed equivariant connections  over flow-outs  of cross-sections, we may use an $LG\times G$-invariant partition of unity, as in Section \ref{subsec:pathfib}, to patch to 
global connections over all of $\P G$.  
\end{proof}
The connections $\alpha,\beta$ give  $LG\times G$-equivariant isomorphisms 
\[ p^*\ol{T}\A\times \g\ \cong\ \ol{T}\P G\ \cong\  q^*TG \times \ol{L}\g.\]
These have the useful property that for all $\gamma\in \PG$, the tangent space $T_\gamma V_\gamma$ to the slice is  contained in the `horizontal space' for the connection. In fact, there is an open neighborhood of  $\gamma$ inside $V_\gamma$ such that $T_{\gamma_1}V_\gamma$ is horizontal for all $\gamma_1$ in that open neighborhood.

\subsection{The twisted $\Spin_c$-structure for general quasi-Hamiltonian spaces}\label{subsec:twistedgeneral}
The construction for conjugacy classes generalizes to arbitrary compact quasi-Hamiltonian $G$-spaces. Unlike the case of conjugacy classes, the construction depends on some  choices, but these choices and the resulting twisted $\Spin_c$-structure are unique up to homotopy. 
\begin{theorem}\label{th:generalham}
Let $G$ be a compact simply connected Lie group, and $\M$ be a proper Hamiltonian $LG$-space, with associated quasi-Hamiltonian 
$G$-space $M$. Consider the correspondence diagram \eqref{eq:corrn}. 
Then there is a spinor module 
\[  \C l(q^*TM)\circlearrowright \Sz_{q^* TM},\]
equivariant with respect to the action of the spin-central extension of $\wh{LG}^{\on{spin}}\times G$. 
The spinor module is canonically defined, up to equivariant  homotopy. 
\end{theorem}
\begin{proof}
Pick an $LG$-invariant compatible complex structure $J\in \ca{J}_{\on{res}}(\ol{T}\M,\omega)$, within the polarization class from Theorem \ref{th:completedtangentbundle}, and let $\Sz_{\ol{T}\M}$ be its spinor module \eqref{eq:thespinormodule}. As in the case of orbits, we consider the splittings 
\[  p^*\ol{T}\M\times (\g\oplus\g)\ \cong\ \ol{T}\N\times \g 
\cong\  q^*TM \times (\ol{L}\g\oplus \g)\]
defined by $\alpha,\beta$. We obtain a Riemannian metric and orthogonal complex structure on $\ol{T}\N\times \g$, 
and a spinor module 
\[ \Sz_{ \ol{T}\N\times\g }
\cong p^*\Sz_{\ol{T}\M}\otimes (\N\times \wedge \g^\C)
\]
over $\Cl(\ol{T}\N\times \g)$. To `descend' under $q$, 
we need an $LG\times G$-equivariant isometric isomorphism 
\[ \ker(\ol{T}q)\cong \N\times \ol{L}\g\stackrel{\cong}{\longrightarrow} \N\times {\bf L\g}.\] 

The square root of the operator $D_\mu,\ \mu=\Phi(m)$ from the construction for orbits gives such an isometric isomorphism \emph{pointwise}, but the resulting bundle map is not smooth, due to the fact that rank of the kernel of $\partial_\mu$ need not be constant. 
To get around this problem, choose a strictly positive function $\chi\in C^\infty(\R)$ such that  $\chi(t)=|t|$ for $t$ outside 
some interval $(-\eps,\eps)$.  Then $\chi(\partial_\mu)$ (defined using the functional calculus)  differs from $D_\mu$ by a finite rank operator. The collection of $LG_\mu$-equivariant isomorphisms 
\begin{equation}\label{eq:chmu}\chi(\partial_\mu)^{1/2}\colon  \ol{L}\g\to {\bf L\g},
\end{equation}
defines an  $LG\times G$-equivariant bundle isomorphism 
\[ \P G\times \ol{L}\g\to \P G\times  {\bf L\g},\]
given at $\gamma$ by \eqref{eq:chmu} with $\mu=p(\gamma)$.  
Hence, by pull-back it gives an equivariant  
bundle isomorphism 
\begin{equation} \label{eq:bisom}\ker(\ol{T}q)
\to 
\N\times {\bf L\g}.
\end{equation}
By Lemma \ref{lem:intertwines} below, this bundle isomorphism \eqref{eq:bisom} intertwines metrics, up to Hilbert-Schmidt equivalence. The construction from Section \ref{subsec:spinor} modifies it further to an equivariant isomorphism exactly intertwining the metrics.  
%
Hence, 
$\Cl({\bf L\g}\oplus \g)$ acts on $\Sz_{\ol{T}\N\times\g}$, and applying Proposition \ref{prop:subspace},
\[ \Sz_{q^* TM}:=\Hom_{\Cl({\bf L\g}\oplus \g)}(\Sz_{{\bf L\g}\oplus \g},
\Sz_{\ol{T}\N\times\g})
\]
is a well-defined  $\wh{LG}^{\on{spin}}\times G$-equivariant spinor module over 
the Clifford algebra of $q^*T M$. 
The choices made in the construction of the spinor module $\Sz_{q^* TM}$ are:  the choice of an $\omega$-compatible complex structure on $\ol{T}\M$ (within the  equivalence class described in Theorem \ref{th:completedtangentbundle}), of 
splittings of the maps $\ol{T}p$ and $\ol{T}q$ as in Theorem \ref{th:shasp}, and of 
a cut-off function $\chi$.  All of these choices are unique up to homotopy. 
\end{proof}

It remains to prove: 

\begin{lemma}\label{lem:intertwines}
The bundle isomorphism \eqref{eq:bisom} intertwines metrics, up to Hilbert-Schmidt equivalence.
\end{lemma}
\begin{proof}
We argue locally, near any given point $n\in \ca{N}$. Let $\gamma=\Phi_{\N}(n)$ and $\mu=\Phi(m)$, let $V_\gamma\subset \P G$ and $V_\mu\subset \A$ be the slices through 
these points,  as in Section \ref{subsec:pathfib}, and denote by $Y=\Phi^{-1}(V_\mu)\subset \M$ the symplectic cross-section through $m=p(n)$. 

By Lemma \ref{lem:hsequivalent}, any two choices of compatible complex structures on $\ol{T}\M$, within the polarization class  from Theorem \ref{th:completedtangentbundle}, give rise to Hilbert-Schmidt equivalent metrics on $\ol{T}\M$. The resulting metric on $\ol{T}\N$ also depends on the choice of $\alpha$; but the bundle map relating the 
splittings $\ol{T}\N\cong\ol{T}\M\times\g$ from any two choices of $\alpha$ differs 
from the identity by a finite rank bundle map; hence so do the $g^\flat$ maps. Hence, 
the Hilbert-Schmidt equivalence class of the metric on $\ol{T}\N$ is 
independent of the  choices made. In particular, we may take the complex structure $J$  adapted to the cross-section $Y$ (cf.~ the proof of Theorem  
  \ref{th:completedtangentbundle}), in the sense that 
it  is  the sum of an $LG_\mu$-invariant compatible complex structure on $TY$ and the standard complex structure  $J_\mu$ (cf.~ \eqref{eq:jmu}) on $\ol{L}\g_\mu^\perp$. The resulting metric on the $\ol{L}\g_\mu^\perp$ summand, at $y\in Y$ with image $
\nu=p(y)\in V_\mu$, is given by  
\begin{equation}\label{eq:firstterm}
 (\xi,\zeta)\mapsto D_\nu \xi\cdot\zeta.
 \end{equation} 
We may furthermore assume that the connection $\alpha$  is defined as in the proof of 
Proposition \ref{prop:splittings}, using the map $R_\mu\colon \ol{L}\g\to \g$. 
At  $z\in \Psi^{-1}(V_\gamma)$, with image point $y\in Y$, the splitting identifies $\ol{T}_z\N=\ol{T}_y\M\times \g$, the inclusion of  $\ker(\ol{T}q)_z=\ol{L}\g\to \ol{T}_z\N$ is given by 
$ \xi\mapsto (\xi_\M(y),\ R_\mu(\xi))$, and the metric   is therefore given by 
\begin{equation}\label{eq:metric1}
 (\xi,\zeta)\mapsto g_\M(\xi_\M,\, \zeta_\M)\big|_y+R_\mu(\xi)\cdot R_\mu(\zeta).\end{equation}
If $\xi,\zeta\in \ol{L}\g_{\mu}^\perp\subset \ol{L}\g$, the second term in \eqref{eq:metric1} does not contribute, while the first term is given by 
\eqref{eq:firstterm}.  We conclude that the given metric on $\ker(\ol{T}q)_z=\ol{L}\g$ 
 differs from the metric \eqref{eq:firstterm} only on a finite-dimensional 
subspace. On the other hand, the metric induced by \eqref{eq:bisom}
reads as 
\begin{equation}\label{eq:metric2}
(\xi,\zeta)\mapsto \chi(\partial_{\nu})\xi\cdot\zeta\end{equation}
Since $\chi(\partial_{\nu})-D_{\nu}$ has finite rank everywhere, we conclude that the two metrics agree on a subspace of finite codimension. 
\end{proof}

\subsection{The canonical line bundle}
Since the Clifford bundle $\C l(q^*TM)$ has finite rank, its spinor module $\Sz_{q^*TM}$ has a well-defined \emph{canonical line bundle} 
\[ \ca{K}=\Hom_{\Cl(q^*TM)}(\Sz_{q^* TM},\ \Sz_{q^* TM}^{\on{op}})\to \N\]
where the superscript ``$\on{op}$'' signifies the opposite (or \emph{dual}) Clifford module. This line bundle is equivariant for the action of the spin-central extension, in such a way that the central circle acts with  weight $-2$. The bundle $\ca{K}$ restricts to an equivariant line bundle over $\M$ (using the inclusion $\M\hra \N$ defined by the inclusion $\A\hra \P G$ as based paths). 
If the compact Lie group $G$ is  simple and simply connected, with dual Coxeter number $\cox$, and letting $\wh{LG}$ denote the \emph{basic} central extension of its loop group, 
then $\ca{K}$ is $\wh{LG}$-equivariant at level $-2\cox$. The canonical line bundle over a Hamiltonian loop group space had been constructed in \cite{me:can} in terms of cross-sections. For the case of coadjoint orbits of the loop group, such line bundles were first discussed in \cite{fre:geo}.

\subsection{Morita morphisms}\label{sec:morita}
The action of $\wh{LG}^{\on{spin}}$ on the 
Hilbert space $\Sz_{\bf{L\g}\oplus \g}$ descends to an action of $LG$ on 
the algebra of compact operators, $\K(\Sz_{\bf{L\g}\oplus \g})$. We define 
a $G$-equivariant Dixmier-Douady bundle
\[ \mathsf{A}^{\on{spin}}=\P G\times_{LG}\K(\Sz_{\bf{L\g}\oplus \g})^{\on{op}}\to \PG/LG=G\]
where the superscript stands for the opposite algebra. Suppose that $(\M,\omega,\Phi)$ is a proper Hamiltonian $LG$-space, and consider the 
factorization 
\begin{equation}\label{eq:spinorproduct} \Sz_{\ol{T}\N\times\g}\cong \Sz_{q^*TM}\otimes (\N\times \Sz_{\bf{L\g}\oplus \g}).\end{equation}
 of the spinor module over $\N$. We may regard the  $LG\times G$-equivariant spinor module $\Sz_{\ol{T}\N\times\g}$ as a bimodule, with $\Cl(q^* TM)$ acting by left multiplication and $\N\times \K(\Sz_{\bf{L\g}\oplus \g})^{\on{op}}$ acting by right multiplication.  Taking quotients by $LG$, this gives a $G$-equivariant Morita bimodule
\[  \Cl(TM) \ \circlearrowright\ \mathsf{E}\ \circlearrowleft\ \Phi^*\mathsf{A}^{\on{spin}}.\]
where $\mathsf{E}=\Sz_{\ol{T}\N\times\g}/LG$. In the terminology of \cite{al:ddd}, it is a Morita morphism 
\[ 
\xymatrix{ \Cl(TM) \ar@{-->}[r]\ar[d] &\  \mathsf{A}^{\on{spin}}\ar[d] \\
M \ar[r]_\Phi & G
}
\] 
This Morita morphism was constructed in \cite{al:ddd} using a completely different approach. 

\section{Abelianization}\label{sec:abel}
In this section, $G$ is a compact, simply connected Lie group, with a fixed 
maximal torus $T$. We show that if the moment map of a proper Hamiltonian $LG$-space is transverse to $\t^*\subset \A$ (as a space of constant connections valued in the Lie algebra of the maximal torus),  then the pre-image $X=\Phi^{-1}(\t^*)$ inherits a $T$-equivariant $\Spin_c$-structure which is also equivariant under a central extension of the lattice $\Lambda\subset \t$. We will also explain how to deal with the non-transverse case. Before discussing the infinite-dimensional setting, we will review the known construction for Hamiltonian $G$-spaces.

\subsection{The spinor module $\Sz_{\g/\t}$}\label{subsec:sgt}
Let $T\subset G$ be a maximal torus , with normalizer $N(T)$ and Weyl group $W=N(T)/T$. We denote by $\mf{R}\subset \t^*$ be the (real) roots $\alpha$ of $(G,T)$, and let $\mf{R}_+$ be the set of of positive roots relative to some choice of Weyl chamber $\t_+\subset \t$. It determines a $T$-invariant complex structure on $\g/\t\cong \t^\perp$, in such a way that the $+\sqrt{-1}$ eigenspace $\n_+=(\g/\t)^+\subset (\g/\t)^\C$ is the direct sum of the root spaces for the positive roots. Let 
\[ \Sz_{\g/\t}=\C l(\g/\t)/\Cl(\g/\t)\,\n_-\cong \wedge \n_+\]
be the associated $T$-equivariant spinor module over $\C l(\g/\t)$. 
The $N(T)$-action on $\g/\t$ does not preserve the complex structure, hence 
does not give an action on the spinor module. The set of `implementers' of this action defines a central extension of $N(T)$ by $\U(1)$, with an action on the 
spinor module such that the central circle acts with weight $1$.  Equivalently, this central extension is the pull-back of the central extension 
\[ 1\to \U(1)\to \on{Pin}_c(\g/\t)\to \on{O}(\g/\t)\to 1\]
under the action of $N(T)$ on $\g/\t$. The $T$-action on the spinor module defines a
homomorphism from $T$ into this central extension; we will identify $T$ with its image in 
$\wh{N(T)}$. The following is well-known. 
\begin{lemma}
Let $g\in N(T)$, with lift $\wh{g}\in\wh{N(T)}$. Then 
\begin{equation}\label{eq:commweil}
  \wh{g}^{-1}t\wh{g} =t^{\rho-w\rho}\ w^{-1}(t),\end{equation}
where $w\in W$ is the Weyl group element determined by $g$, and  $t^{\rho-w\rho}\in \U(1)$ is the image of $t$ under the homomorphism 
$T\to \U(1)$ defined by the weight $\rho-w\rho$. 
\end{lemma}
\begin{proof}
Since the left hand side of \eqref{eq:commweil} is a lift of $g^{-1}tg=w^{-1}(t)$, it 
differs from the `canonical' lift by a scalar. To determine this scalar, let us apply both sides to the `vacuum vector' $1\in \wedge\n_+$. This element is annihilated by the Clifford action of all root vectors $e_\alpha$ with $\alpha\in 
\mf{R}_-$, hence $\wh{g}.1$ is annihilated by all $g.e_\alpha$ with $\alpha\in \mf{R}_-$. These are the root vectors for the weights $w\alpha$, and the pure spinor line annihilated by these root vectors is spanned by the wedge product of root vectors for roots $\beta\in\mf{R}_+$ such that $w^{-1}\beta\in \mf{R}_-$. The weight for the action of $t$ on this wedge product is thus the sum over all positive roots $\beta$ for which $w^{-1}.\beta$ is negative. But this is the weight $\rho-w\rho$. 
\end{proof}
The spinor module has a $\Z_2$-grading $\Sz_{\g/\t}^{ev} \oplus \Sz_{\g/\t}^{odd}$, and the $\Z_2$-graded character 
\[ \Delta(t)=\on{tr}\big(t\big|_{\Sz_{\g/\t}^{ev}}\big) 
-\on{tr}\big(t\big|_{\Sz_{\g/\t}^{odd}}\big) \]
is complex conjugate to the Weyl denominator:
\[ \Delta(t)^*=\prod_{\alpha\in\mf{R}_+}(1-t^{-\alpha})=\sum_{w\in W}(-1)^{l(w)} t^{w\rho-\rho}.\]

\begin{remark}
The structure group of the central extension $\wh{N(T)}$ can be reduced to $\Z_2$, by taking the pull-back of $\on{Pin}(\g/\t)$ rather than $\on{Pin}_c(\g/\t)$. The  embedding $T\to \wh{N}(T)$ does not take values in this $\Z_2$-central extension. However, if $G$ is simply connected, then $\rho$ is a weight, and one can use a new lift 
\[ \iota\colon T\to \wh{N(T)},\ \ t\mapsto \wh{t}=t^{-\rho}\,t\] 
which takes values in $\on{Pin}$. Equation \eqref{eq:commweil} shows that the image of this map is a normal subgroup, resulting in a  central extension of $W$ by $\Z_2=\pm 1$. 
See \cite{mor:pro} for an explicit description of this central extension in terms of generators and relations. 
\end{remark}

There is a similar discussion for the Clifford algebra over $\mathbf{L}\g/\t$ (identified with the orthogonal space to $\t$ inside $\mathbf{L}\g$). This space has a complex structure whose $+\sqrt{-1}$ eigenspace is spanned by $\n_+$ together with the subspace of $\mathbf{L}\g^{\mathbb{C}}$ spanned by the positive Fourier modes. It defines a 
$T$-equivariant spinor module 
\[ \Sz_{\mathbf{L}\g/\t}=\ol{\wedge (\mathbf{L}\g/\t)^+}\]
where the bar signifies a Hilbert space completion. 

The action of the subgroup $\Lambda\rtimes N(T)\subset LG$, where $\Lambda\subset \t$ is the integral lattice embedded as `exponential paths', preserves $\t$, hence also 
$\mathbf{L}\g/\t$. The latter action is by transformations in the restricted orthogonal group $\on{O}_{res}(\mathbf{L}\g/\t)$, hence we obtain 
a central extension of $\Lambda\rtimes N(T)$ consisting of all implementers of this action.

As before it is convenient to do computations using a basis of root vectors.  There is an action of $T \times S^1$ on $\mathbf{L}\g^{\mathbb{C}}$, where $T$ acts by the adjoint action and $S^1$ acts by rotating the loop.  The non-zero weights for this action are the \emph{affine roots} $\mf{R}_{\on{aff}}$: all pairs $(\alpha,n)$, $\alpha \in \mf{R} \cup \{0 \}$, $n \in \mathbb{Z}$ where either $n \ne 0$, or $n=0$ and $\alpha \ne 0$.  The subset $\mf{R}_{\on{aff},+}$ with either $n>0$ or $n=0$ and $\alpha \in \mf{R}_+$ are the \emph{positive affine roots}; the latter are the weights for the action of $T \times S^1$ on $(\mathbf{L}\g/\t)^+$.  Let $\mf{R}_{\on{aff},-}=-\mf{R}_{\on{aff},+}$ denote the negative affine roots.

Since the $T$-action 
preserves the complex structure, the central extension is canonically trivial over $T$, and given a lift $\wh{g}$ of $g\in \Lambda\rtimes N(T)$, where $g$ maps to the affine Weyl group element $w\in \Lambda\rtimes W$, we have 
\begin{equation}\label{eq:skew}
\wh{g}^{-1}t\wh{g}=t^{\sum'\alpha}\ w^{-1}(t),
\end{equation}
where the sum $\sum'\alpha$ is over all $(\alpha,n) \in \mf{R}_{\on{aff},+}\cap w \mf{R}_{\on{aff},-}$.

If $\g$ is simple then
\[ \sideset{}{'}\sum \alpha=\rho-w\rho. \]
Here $w\rho$ denotes the action of the affine Weyl group on $\t^\ast$ at level the \emph{dual Coxeter number} $\cox$; in terms of the basic inner product $B$ for $\g$, the latter is generated by reflections in the affine hyperplanes $\cox H_{(\alpha,n)}$ where 
\[ H_{(\alpha,n)}=\{\xi \in \t^\ast|B(\alpha,\xi)+n=0 \}\]
and $(\alpha,n) \in \mf{R}_{\on{aff},+}$.
\begin{remark}
Let $\theta$ denote the highest root.  Using the formula
\[ \cox=1+B(\rho,\theta)\]
for the dual Coxeter number, one finds immediately that $\rho-w\rho=-\theta$, for $w$ the reflection in $\cox H_{(-\theta,1)}$.  This verifies the formula $\sum'\alpha=\rho-w\rho$ for the additional generator $w$ of the affine Weyl group.
\end{remark}

\subsection{Abelianization for Hamiltonian $G$-spaces}
In this section, we consider Hamiltonian $G$-spaces $(M,\omega,\Phi)$ whose moment map 
$\Phi\colon M\to \g^*$ is transverse to $\t^*=(\g^*)^T\subset \g^*$. Simple examples of such spaces include $M=T^*G$, with the $G$-action given by the cotangent lift of the left-action of $G$ on itself, or the regular coadjoint orbits $\O\subset \g^*$.  

\begin{remark}
The paper \cite{wo:cl} gives a classification of all multiplicity free Hamiltonian $G$-spaces
whose moment map is transverse to $\t^*$. For example, every regular coadjoint orbit of $\U(3)$, 
regarded as a Hamiltonian $G=\U(2)$-space under the inclusion $U(2)\to U(3)$ in the upper left corner, is such a space. 
\end{remark}

The transversality assumption implies that  the submanifold 
\[ X=\Phi^{-1}(\t^*),\] 
with the pull-backs $\omega_X,\Phi_X$ of the symplectic form and moment map, becomes a degenerate Hamiltonian $N(T)$-space. Here \emph{degenerate} refers to the fact that the  2-form $\om_X$ is no longer symplectic; its has a non-trivial kernel 
at points $x\in X$ for which $\Phi(x)$ is not regular (i.e., has stabilizer larger than $T$). 
There is a canonical $N(T)$-equivariant trivialization of the normal bundle, 
\[ \nu(M,X)\cong  X\times\ \g/\t\] 
coming from the bundle map $\nu(M,X)\to \nu(\g^*,\t^*)\cong (T\t^*)^\perp=\t^*\times \g/\t$ induced by $\Phi$. 
It has a $T$-equivariant $\Spin_c$-structure from the complex structure 
on $\g/\t$, with associated spinor module $X\times \Sz_{\g/\t}$.  

On the other hand, the choice of a $G$-invariant compatible almost complex structure on $M$ defines a $G$-equivariant spinor module $\Sz_{TM}$, and we obtain a 
$T$-equivariant spinor module for $X$,  
\begin{equation}\label{eq:sz}
\Sz_{TX}=\Hom_{\C l(\g/\t)}(X\times \Sz_{\g/\t},\Sz_{TM}|_X).
\end{equation}
Here the metric on $TX^\perp\cong \nu(M,X)$ induced by the metric on $M$ may be different from that coming from the isomorphism with $X\times \g/\t$, but as explained in Section \ref{subsec:spinor} there is a natural isometric isomorphism relating the two metrics. .

Given a $G$-equivariant  line bundle $L\to M$, the equivariant indices 
of $L|_X\otimes \Sz_{TX}$ and of $L\otimes \Sz_{TM}$ are related by the 
Weyl denominator: 
\begin{equation}\label{eq:sz1}
 \on{index}_G(L\otimes \Sz_{TM})(t)=\f{\on{index}_T(L|_X\otimes \Sz_{TX})(t)}{\prod_{\alpha\in\mf{R}_+}(1-t^{-\alpha}) }
 \end{equation}
for all regular $t\in T$. In particular, the two equivariant indices carry the same information. 
\begin{remark}
The $T$-equivariant index $\chi(t)=\on{index}_T(L|_X\otimes \Sz_{TX})(t)$ has the following transformation property. 
\[ \chi(w^{-1}.t)=(-1)^{l(w)}t^{w\rho-\rho}\chi(t).\] 
This follows one the one hand from the invariance properties of the Weyl denominator in 
\eqref{eq:sz1}, while the left hand side is $W$-invariant. 

Alternatively, it  follows because the $T$-action on the spinor module $\ca{S}_{TX}$, extends to an action of the central extension of $N(T)$, see Section \ref{subsec:sgt}: If $h\in N(T)$ represents $w$, and $\wh{g}$ lifts $g$, then 
\[ \chi(w^{-1}.t)=\chi(t^{w\rho-\rho}\ \wh{g}^{-1} t \wh{g})=
(-1)^{l(w)} t^{w\rho-\rho}\,\chi(t).\]
Here the sign appears because $g$ changes the orientation, and hence $\wh{g}$ changes the parity of the spinor module, by $(-1)^{l(w)}$.

Since the complex structure on $\g/\t$ is not $N(T)$-equivariant, the 
spinor module $\Sz_{TX}$ is not $N(T)$-equivariant. However, we obtain a 
central extension of $N(T)$ (given by implementers for the 
action on $\Sz_{\g/\t}$), and this central extension acts.  This accounts for the fact that the resulting index $\on{index}_T(L_X\otimes \Sz_{TX})\in R(T)$ is \emph{anti-invariant} with respect to the shifted Weyl group action. 
\end{remark}
In practice, the assumption that $\Phi$ is transverse to $\t^*$ is rather 
strong. To make the abelianization procedure work in general, one can use an 
$N(T)$-invariant tubular neighborhood $U\subset \g^*$ of $\t^*$, over which an  
`equivariant Bott element' $\beta$ (the $K$-theory counterpart of a Thom form) is defined. Over $\Phinv(U)$, one then takes the cup product of the $K$-homology class of
$\Spin_c$-structure, with the pull-back of $\beta$; this has a well-defined equivariant index. In the transverse case, one can `integrate over the fibers' and replace 
the $K$-cycle on $\Phinv(U)$ with one on $X=\Phinv(\t^*)$. 

\subsection{Abelianization for Hamiltonian $LG$-spaces}\label{subsec:abel}
Our aim is to carry out a similar abelianization approach for Hamiltonian loop group spaces, with moment maps taking values in $\A$.  This works particularly well in the `transverse case'. Let $\Lambda\subset \t$ be the integral lattice (kernel of the exponential map), with the natural action of  $N(T)$ via the homomorphism to $W$. The semi-direct product 
$\Lambda\rtimes N(T)$ acts on $\t$ via its homomorphism to the affine Weyl group $\Lambda\rtimes W$, that is, 
$(\ell,\,h).\mu= \Ad_h\mu-\ell$. The inclusion $\t\to \P G$ as exponential loops $\gamma(t)=\exp(t\mu)$ 
is equivariant  relative to the homomorphism
\[\Lambda\rtimes N(T)\hra LG\times G,\ \  (\ell,h)\mapsto (\exp(t\ell)h, h).\]
Under the projection $p\colon \P G\to \A$, it descends to an equivariant  inclusion $\t\to \A,\ \ \mu\mapsto \mu\partial t$ as `constant connections', and under $q\colon \P G\to G$ it descends to the standard inclusion $T\to G$. 
\begin{proposition}
For a proper Hamiltonian $LG$-manifold $(\M,\om,\Phi)$, with correspondence diagram \eqref{eq:corrn}, the following are equivalent: 
\begin{itemize}
\item the map $\Phi\colon \M\to \A$ is transverse to $\t\hra \A$,
\item the map $\Phi_M\colon M\to G$ is transverse to $T\hra G$, 
\item the map $\Phi_\N\colon \N\to \P G$ is transverse to $\t\hra \P G$. 
\end{itemize} 
\end{proposition}
\begin{proof}
Let $\mu\in \t$ be given, corresponding to the exponential path  $\gamma(t)=\exp(t\mu)$, and let $n\in \N$ with $\Phi_\N(n)=\gamma$. Let $V_\gamma$ be the slice at $\gamma$, and $Y_\gamma=(\Phi_\N)^{-1}(V_\gamma)\subset \N$ the cross-section. Note that $V_\gamma$ contains an open neighborhood of $\gamma$ inside the image of $\t\hra \P G$.  Hence, by equivariance, 
$\Phi_\N$ is transverse to $\t\hra \P G$ at $n$ if and only if its restriction 
$\Phi_\N|_{Y_\gamma}$ is transverse to $\t$ at $n$. 
Similarly, the transversality of $\Phi_\M$ to $\t$ at $p(n)$, and of $\Phi_M$ to $T$ at $q(n)$, is equivalent to a similar transversality conditions for the restriction to the cross-sections. Since $p$ and $q$ define diffeomorphisms of the cross-sections $Y_{p(\gamma)} \leftarrow Y_\gamma\rightarrow Y_{q(\gamma)}$, intertwining the moment maps and the inclusions of $\t$  resp, of $T$, we conclude that all three transversality conditions are equivalent. 
\end{proof}

\begin{example}
There is an interesting example, due to Woodward,  of a \emph{multiplicity-free} quasi-Hamiltonian $\SU(3)$-spaces, with moment map transverse to the maximal torus $T$. 
Its moment polytope is a triangle, with vertices the mid-points of the edges of the alcove. 
See \cite{me:lec} or \cite{loi:nor} for a description of this space. 
\end{example}

\begin{proposition} Let $(\M,\om,\Phi)$ be a proper Hamiltonian $LG$-space. 
If the moment map $\Phi\colon \M\to \A$ is transverse to the inclusion $\t\hra \A$, then 
the pre-image $\ca{X}=\Phinv(\t)$ becomes a (degenerate) Hamiltonian $\Lambda\rtimes N(T)$-manifold. The choice of positive roots for $(G,T)$ determines a $\Spin_c$-structure on $\ca{X}$, equivariant with respect to 
\[ \wh{\Lambda}^{\on{spin}}\rtimes T\]
where the superscript indicates the central extension obtained by restriction from 
the spin central extension of the loop group.   
\end{proposition}
\begin{proof}
By the transversality condition, the pre-image $X=\Phinv(T)$ is a (degenerate) quasi-Hamiltonian $N(T)$-manifold, while 
$\ca{X}=\Phi^{-1}(\t)\subset \M$ becomes a (degenerate) Hamiltonian $\Lambda\rtimes N(T)$-manifold. We may also regard it as the submanifold  $\ca{X}=\Phi_\N^{-1}(\t)\subset \N$. 

The moment map $\Phi_M$ and any choice of invariant Riemannian metric gives an $N(T)$-equivariant decomposition $TM|_X\cong TX\oplus (X\times \g/\t)$, hence
\[ q^*TM|_{\ca{X}}=T\ca{X}\oplus (\ca{X}\times \g/\t)\]
Using the spinor bundle $\Sz_{\g/\t}$ defined by a system of positive roots, we obtain a spinor bundle, 
\[ \Sz_{T\ca{X}}=\Hom_{\Cl(\ca{X}\times \g/\t)}(\ca{X}\times \Sz_{\g/\t},\Sz_{q^*TM}|_{\ca{X}}), 
\]
hence a \emph{$\Spin_c$-structure} on $\ca{X}$. 
\end{proof}
\begin{remark}
From the proof, we see that the spinor module is actually equivariant for the action of 
a central extension of $\Lambda\rtimes N(T)$ whose restriction to $N(T)$ is the opposite of the spin-central extension for its action on $\Sz_{\g/\t}$. 
\end{remark}
\begin{remark}
In this argument, we used the twisted $\Spin_c$-structure on $q^* TM$ constructed earlier. Alternatively, one can also start out with $\ca{S}_{\ol{T}\M}$. Similar to Section \ref{subsec:twistedgeneral}, one can construct an $\Lambda\rtimes N(T)$-equivariant isometric  isomorphism 
\[ \ol{T}\M=T\ca{X}\oplus \bf{L}\g/\t,\]
and then define $\Sz_{T\ca{X}}=\Hom_{\Cl(\ca{X}\times \bf{L}\g/\t)}(\ca{X}\times \Sz_{\bf{L}\g/\t}, \ \Sz_{\ol{T}\M})$. 
\end{remark}

\begin{remark}
Given a pre-quantum line bundle $\ca{L}$ for $\M$, one obtains a new $\Spin_c$-structure on $\ca{X}$, with spinor module $\Sz_{T\ca{X}}\otimes \ca{L}$, and an associated Dirac operator $\dirac$. This Dirac operator is equivariant with respect to 
$T$, as well as with respect to a spin-central extension of $\Lambda$. Using the commutation relations between these two actions, we can show that the irreducible 
$T$-representations appear with finite multiplicity in the (infinite-dimensional) 
kernel and cokernel of $\dirac$. Details will be given in a forthcoming paper. 
\end{remark}

\subsection{Thickenings}\label{subsec:thickening}
In the non-transverse case, the situation is slightly more complicated. Choose $\eps>0$ so that the map 
\[ T\times \g/\t\cong \nu(G,T)\to G,\ \ (t,\xi)\mapsto t\exp(\xi)\]
(where we identify $\g/\t\cong \t^\perp$) restricts to a tubular neighborhood embedding
\begin{equation}\label{eq:F1}
T\times B_\epsilon(\g/\t)\hra G.
\end{equation} 
Using an $LG\times G$-equivariant principal connection on $\P G\to G$ (for instance, the connection $\beta$ constructed in Section \ref{subsec:conn}), this lifts uniquely to a tubular neighborhood embedding 
\begin{equation}\label{eq:F2}  q^{-1}(T)\times B_\epsilon(\g/\t)\hra \P G
\end{equation}
in such a way that the corresponding Euler vector field on the image of \eqref{eq:F2}
is the horizontal lift of the Euler vector field on the image of \eqref{eq:F1}. See \cite{bur:spl} for the construction of tubular neighborhoods in terms of Euler-like vector fields; the relevant infinite-dimensional techniques can be found in \cite{ab:ma}.
Its composition with the inclusion $\t\to q^{-1}(T)\subset \P G$ as exponential maps
defines a $\Lambda\rtimes N(T)$-equivariant embedding $\t \times B_\epsilon(\g/\t)\hra \P G$, which fits into a commutative diagram
\begin{equation}\label{eq:F3}  
\xymatrix{ \t\times B_\eps(\g/\t)\ar[r]\ar[d]& \P G \ar[d]\\
T \times B_\eps(\g/\t)\ar[r] & G
}
\end{equation}
Since the upper map is transverse to the action of $LG$, the pre-image 
\[ \ca{Y}=\Phi_\N^{-1}(\t\times B_\eps(\g/\t))\subset \N\]
is a smooth finite-dimensional submanifold. We think of it as a thickened version of the possibly singular space $\ca{X}=\Phi_\N^{-1}(\t)\cong \Phi^{-1}(\t)$. Projection to 
$M$ gives an identification 
\[ \ca{Y}/\Lambda\cong Y:=\Phi_M^{-1}(T \times B_\eps(\g/\t))\subset M,\]
an open neighborhood of $X$ in $M$. The restriction of $q^*\Sz_{TM}$ to $\ca{Y}$ 
is a $\wh{\Lambda}^{\on{spin}}\rtimes T$-equivariant spinor bundle, defining a 
$\Spin_c$-structure on $\ca{Y}$.

\section{Twisted loop groups}\label{sec:twisted}
Let $\kappa\in \on{Aut}(G)$ be a Lie group automorphism. We take the 
\emph{$\kappa$-twisted loop group} $L^{(\kappa)}G$ to be the 
group of paths $\lambda\colon \R\to G$ of (local) Sobolev class $s$ with the property $\lambda(t+1)=\kappa(\lambda(t))$. 

Let $\P^{(\kappa)}G$ consist of paths $\gamma\colon \R\to G$ of Sobolev class $s$ such that 
$\gamma(t+1)=a\kappa(\gamma(t))$ for some $a\in G$. The group $G$ acts on 
$\P^{(\kappa)}G$ by multiplication from the left, while 
$L^{(\kappa)} G$ acts as multiplication by the inverse from the right. We will use the notation $G\kappa$ for the group $G$ 
regarded as a $G$-space under the   \emph{$\kappa$-twisted} conjugation action,
$g.a=ga\kappa(g)^{-1}$. Then $\P^{(\kappa)}G$ is a 
$G$-equivariant principal $L^{(\kappa)} G$-bundle over $G\kappa$, with quotient map
$q(\gamma)=\gamma(1) \kappa(\gamma(0))^{-1}$. Let $ \A^{(\kappa)}$ be the space of connection 1-forms on $\R$ of Sobolev class $s-1$, with the property $\mu(t+1)=\kappa(\mu(t))$. The group $L^{(\kappa)}G$ acts on this space by gauge transformations, and the map $p\colon \P^{(\kappa)}G\to \A^{(\kappa)},\ \ \gamma\mapsto \gamma^{-1}\partial\gamma$ is an $L^{(\kappa)}G$-equivariant principal $G$-bundle. We arrive at the correspondence diagram, 
\begin{equation}\label{eq:corrkappa}
 \xymatrix{ & \P^{(\kappa)} G\ar[dl]_p\ar[dr]^q & \\
\mathcal{A}^{(\kappa)} & & G\kappa
}\end{equation}
giving a Morita equivalence of the action groupoids $[\A^{(\kappa)}/L^{(\kappa)}G]$ and $[G\kappa/G]$. Given an $\Ad$-invariant metric on $\g$ that is also $\kappa$-invariant, we obtain an 
$L^{(\kappa)}G\times G$-invariant  
2-form 
$\varpi^{(\kappa)}\in \Omega^2(\P^{(\kappa)}G)$, with the properties 
$\d\varpi^{(\kappa)}=-q^*\eta$, as well as 
\begin{equation}\label{eq:twistmom}
\iota(X_{\PG}) \varpi^{(\kappa)}=-\hh q^*(\theta^L\cdot \kappa(X)+\theta^R\cdot X),\ \ \ \iota(\xi_{\PG})\varpi^{(\kappa)}=p^*\l\d\mu,\xi\r\end{equation}
for $X\in \g,\ \xi\in L^{(\kappa)}\g$. It is given by the explicit formula (see Appendix \ref{app:2form})
\[ \varpi^{(\kappa)}=\hh \int_0^1 \Big(\on{ev}_t^*\theta^R\cdot \f{\p}{\p t}  \on{ev}_t^*\theta^R\Big)\partial t
+ \hh \on{ev}_0^*\kappa(\theta^L)\cdot \on{ev}_1^*\theta^L.
\]
Proper Hamiltonian $L^{(\kappa)}G$-spaces are defined just as in the case of a trivial automorphism (see Definition \ref{def:melo}), replacing $\A$ with $\A^{(\kappa)}$. 
Similarly, quasi-Hamiltonian $G$-spaces with $G\kappa$-valued moment maps 
\cite{boa:twi,me:conv}
are defined 
similar to Definition \ref{def:quas}, but requiring equivariance with respect to the $\kappa$-twisted conjugation action, and replacing the moment map condition by 
\[ \iota(X_M)\omega_M=-\hh \Phi_M^*(\theta^L\cdot \kappa(X)+\theta^R\cdot X),\] 
in accordance with \eqref{eq:twistmom}. 
As before, one has a 1-1 correspondence between proper Hamiltonian $L^{(\kappa)}G$-spaces and quasi-Hamiltonian $G$-spaces with $G\kappa$-valued moment maps, described by a diagram $\M\stackrel{p}{\longleftarrow} \N \stackrel{q}{\longrightarrow} M$, similar to \eqref{eq:corrn}.  Examples of such quasi-Hamiltonian spaces are the twisted conjugacy classes $\Co\subset G\kappa$; the corresponding $L^{(\kappa)}G$-spaces are the coadjoint $L^{(\kappa)}G$-orbits $\O\subset \A^{(\kappa)}$. Here, the symplectic form on coadjoint orbits is given by the same expression \eqref{eq:coadsympl} as in the untwisted case, using that $\partial_\mu\xi_1\cdot\xi_2$ for $\mu\in \A^{(\kappa)}$ and $\xi_i\in L^{(\kappa)}\g$ is a periodic 1-form on $\R$. Further examples may be found in the context of twisted wild character varieties \cite{boa:twi}, twisted moduli spaces of flat connections \cite{me:conv}, and multiplicitity free quasi-Hamiltonian spaces \cite{kno:mul}.

\begin{remark}\label{rem:enough}
For most purposes, it is enough to consider one representative of automorphisms in any given class in $\on{Aut}(G)/\on{Inn}(G)$. (See \cite[Section 3.4]{me:conv}.) Indeed, suppose $\wt{\kappa}=\Ad_h \circ \kappa$ for some $h\in G$. Then the map  
$G\kappa\to G\wt{\kappa},\ \ g\mapsto gh^{-1}$ intertwines the $\kappa$-twisted conjugation action with the $\wt{\kappa}$-twisted conjugation action. 
Similarly, the choice of any 
$\sigma\in \P^{(\kappa)}G$ with $q(\sigma)=h$ defines a group isomorphism 
$L^{(\kappa)}G\to L^{(\wt{\kappa})}G,\ \lambda\mapsto \Ad_{\sigma}\lambda$. 
%
The map 
$\P^{(\kappa)}G\to P^{(\wt{\kappa})}G,\ \gamma\mapsto \gamma\sigma^{-1}$,
is $G$-equivariant and intertwines the actions of $ L^{(\kappa)}G$ and $L^{(\wt{\kappa})}G$. The corresponding map $\A^{(\kappa)}\to \A^{(\wt{\kappa})}$ is given by $\mu\mapsto \Ad_\sigma(\mu)-\partial\sigma \sigma^{-1}$. 
In this way, right multiplication by $h^{-1}$ turns a $G\kappa$-valued moment map into a 
$G\wt{\kappa}$-valued moment map, and the gauge action of $\sigma$ turns an 
$\A^{(\kappa)}$-valued moment map into a $\A^{(\wt{\kappa})}$-valued moment map.
\end{remark}

Given a proper Hamiltonian $L^{(\kappa)}G$-space $\M$, the construction of a spinor module $\Sz_{\ol{T}\M}$ proceeds parallel to the case 
of $\kappa=1$. We will be brief, providing details only where special aspects of the construction arise. We introduce the space $\ol{L}^{(\kappa)}\g$ of $\kappa$-twisted loops of Sobolev class $\f{1}{2}$, and use it to define a completion $\ol{T}\M$ of the tangent bundle, on which the 2-form becomes strongly symplectic. If $\O\subset \A^{(\kappa)}$ is a coadjoint orbit, then the completed tangent bundle $\ol{T}\O$ has a canonical $L^{(\kappa)}G$-invariant compatible complex structure, given by 
the formula \eqref{eq:jmu}. In the general case, we obtain a polarization class of $L^{(\kappa)}G$-invariant compatible complex structures $J$ on $\ol{T}\M$. Up to isomorphism, the resulting  $L^{(\kappa)}G$-equivariant spinor bundle 
$\Sz_{\ol{T}\M}$ is independent of the choice of $J$.

To obtain a twisted $\Spin_c$-structure for the associated quasi-Hamiltonian space (where `twist' refers to the twisting of $K$-theory, rather than to the automorphism $\kappa$ of the twisted loop group), we need  $LG\times G$-equivariant connections $\alpha,\beta$ on the two principal bundles in  \eqref{eq:corrkappa}, in such way that the corresponding vertical projections extend to the completions. Such connections are obtained by the same method as in Section \ref{subsec:conn}; see \cite{me:conv} for a discussion of slices for the twisted conjugation action. As another ingredient, we  need the spin representation of the twisted loop group. Let $\bf{L}^{(\kappa)}\g$ be the 
$\kappa$-twisted loops of Sobolev class $0$, with the 
Hilbert space inner product given by integration over $[0,1]\subset \R$.  The covariant derivative $\partial_0$ with respect to $0\in \A^{(\kappa)}$ is an unbounded skew-adjoint operator on this Hilbert space, its kernel are the constant $\kappa$-twisted loops, that is, elements of  $\g^\kappa$. 
The spectral decomposition of $\partial_0$ defines 
a complex structure on $(\g^\kappa)^\perp\subset \bf{L}^{(\kappa)}\g$; together with the standard complex structure on $\g^\kappa\oplus \g^\kappa$ 
we hence have a complex structure on $\bf{L}^{(\kappa)}\g\oplus \g^\kappa$, and an associated spinor module:
\[ \Cl(\bf{L}^{(\kappa)}\g\oplus \g^\kappa)
\circlearrowright \Sz_{ \bf{L}^{(\kappa)}\g\oplus \g^\kappa }.\]
We obtain a central extension of $L^{(\kappa)}G$ by its map 
to the restricted orthogonal group; its opposite will be called the 
spin-central extension of the twisted loop group $L^{(\kappa)}G$.

Consider the principal bundles $\M\stackrel{p}{\longleftarrow} \N \stackrel{q}{\longrightarrow} M$, obtained by pulling back \eqref{eq:corrkappa} under the respective moment maps, and equipped with the pull-backs of the connections $\alpha,\beta$. We obtain $LG\times G$-equivariant isometric bundle isomorphisms
\[ p^*\ol{T}\M\times (\g\oplus \g)\cong \ol{T}\N \times \g
\cong q^* TM\times( \bf{L}^{(\kappa)}\g\oplus \g)
\]
(using a trivial action on the second $\g$ copy).  Here the second isomorphism is obtained by first using $\beta$ to identify 
$\ol{T}\N \times \g$ with the bundle 
$q^* TM\times ( \ol{L}^{(\kappa)}\g\oplus \g)$, and then using the method from Section 
\ref{subsec:twistedgeneral} to pass to $\bf{L}^{(\kappa)}\g$, in such a way that metrics and  actions are preserved. The first isomorphism gives an 
equivariant spinor bundle 
\[ \Sz_{\ol{T}\N\times  \g}:=p^*\Sz_{\ol{T}\M}\otimes (\N\times \wedge\g^\C)\]
Taking a `quotient', we obtain a  spinor module 
\[q^*\Cl(TM\oplus \g/\g^\kappa) \circlearrowright \Sz_{q^*TM\times \g/\g^\kappa}:=\Hom_{\C l(\bf{L}^{(\kappa)}\g\oplus \g^\kappa)}\big( \Sz_{ \bf{L}^{(\kappa)}\g\oplus \g^\kappa },\ 
\Sz_{\ol{T}\N\times\g}\big);\]
equivalently, we obtain a $\Spin_c$-structure on the bundle
\begin{equation}\label{eq:thebundle}
q^*TM\times \g/\g^\kappa\to \N,
\end{equation}
which is equivariant for the action of $\wh{L^{(\kappa)}G}^{\on{spin}}\times G$.    The presence of the $\g/\g^\kappa$ factor is both natural and convenient.  In fact, quasi-Hamiltonian spaces with $G\kappa$-valued moment maps can be odd-dimensional; examples include the twisted conjugacy classes of $\SU(3)$ with $\kappa$ defined by the standard diagram automorphism. Using the  cross-section theorems from \cite{me:conv}, one finds that the parity of $\dim M$ coincides with that of $\dim \g/\g^\kappa$, hence $q^*TM\times \g/\g^\kappa$ always has even rank. 

We define a $G$-equivariant Dixmier-Douady bundle over $G\kappa$, 
\[ \mathsf{A}^{(\kappa),\on{spin}}=\P^{(\kappa)}G\times_{L^{(\kappa)}G}
\K(\Sz_{\bf{L^{(\kappa)}\g}\oplus \g^\kappa})^{\on{op}}\to G\kappa.\]
The reasoning from Section \ref{sec:morita} gives a $G$-equivariant Morita morphism, 
\[ \Cl(q^*TM\times \g/\g^\kappa)\dasharrow \mathsf{A}^{(\kappa),\on{spin}}.\]
The K-homology fundamental class of $M$ lives in the $G$-equivariant twisted K-homology  of $M$ with coefficients 
in $\Cl(q^*TM\times \g/\g^\kappa)$; taking a tensor product with a pre-quantization, and pushing forward under the moment map as in \cite{me:twi}
gives an element of twisted equivariant $K$-homology of $G\kappa$, at a suitable level. 

For the abelianization procedure, we assume that $G$ is simple and simply connected, and (with no loss of generality, see Remark \ref{rem:enough}) that $\kappa$ is given by a Dynkin diagram automorphism, relative to some choice of maximal torus $T$ and positive roots. Then $\kappa$ preserves $T$; let $T^\kappa$ be the fixed point torus. Every $\kappa$-twisted conjugacy class meets $T^\kappa$; similarly, every coadjoint orbit of $L^{(\kappa)}G$ meets $\t^\kappa\subset \A^{(\kappa)}$ (embedded as constant connections), and every orbit of 
$L^{(\kappa)}G\times G$ on $\P^{(\kappa)}G$ meets $\t^\kappa\subset \P^{(\kappa)}G$ (embedded as exponential paths). 
 
Suppose $\M$ is a proper Hamiltonian $L^{(\kappa)}G$-space whose moment map $\Phi\colon \M\to \A^{(\kappa)}$ is transverse to $\t^{\kappa}\subset \A^{(\kappa)}$. Equivalently, $\Phi_\N$ is transverse to $\t^\kappa\subset \P^{(\kappa)}G$, and $\Phi_M$ is transverse to $T^\kappa\subset G\kappa$.  Then  $\mathcal{X}=\Phi^{-1}(\t^{\kappa})\cong \Phi_\N^{-1}(\t^\kappa)$ is a degenerate Hamiltonian $\Lambda^\kappa\rtimes N_G(T)^{\kappa}$-space, while $X=\Phi_M^{-1}(T^{\kappa})$ is a degenerate quasi-Hamiltonian $N_G(T)^{\kappa}$-space.  We have 
$TM|_X=TX\oplus (X\times \g/\t^\kappa)$, and a similar decomposition for the pull-back to $\ca{X}\subset \N$. Consequently, 
\[ q^*TM|_{\ca{X}}\oplus (\ca{X}\times \g/\g^\kappa)
=T\ca{X}\oplus (\ca{X}\times (\g/\t^\kappa\oplus \g/\g^\kappa)),\]
We obtain a complex structure on $\g/\t^\kappa\oplus \g/\g^\kappa
=\g^\kappa/\t^\kappa\oplus  \g/\g^\kappa\oplus  \g/\g^\kappa$, by taking the 
standard complex structure on $\g/\g^\kappa\oplus  \g/\g^\kappa=
\g/\g^\kappa\otimes \R^2$, and the complex structure on $\g^\kappa/\t^\kappa$ determined by the positive roots. Together with the $\Spin_c$-structure on 
$q^*TM|_{\ca{X}}\oplus (\ca{X}\times \g/\g^\kappa)$, this then determines a 
$\Spin_c$-structure on $\ca{X}$, equivariant under the action of the spin central extension of $\Lambda^\kappa\rtimes N_G(T)^{\kappa}\subset L^{(\kappa)}G$. 
The non-transverse case can be dealt with by a thickening procedure similar to 
Section \ref{subsec:thickening}.

\begin{remark}
The $\Spin_c$-structure on $\ca{X}$ is equivariant for 
the spin-central extension of the full subgroup of $L^{(\kappa)}G$ preserving 
$\t^\kappa\subset \A^{(\kappa)}$. This is  somewhat larger than $\Lambda^\kappa\rtimes N_G(T)^\kappa$. To see this,  let $T_\kappa$ be the range of 
$T\to T,\ h\mapsto f(h):=h\kappa(h)^{-1}$ given by the twisted conjugation action on the group unit $e$. The subtorus $T_\kappa$ is transverse to $T^\kappa$, of complementary dimension; hence $T_\kappa\cap T^\kappa$ is a finite group.  The twisted conjugation action of $h\in T$ preserves $T^\kappa$ if and only if 
$f(h)\in T_\kappa\cap T^\kappa$ (in which case the action is translation by $f(h)$), 
and $N_G^{(\kappa)}(T^\kappa)$ (the subgroup of $G$ whose action on $G\kappa$ preserves $T^\kappa$)  is generated by $N_G(T)^\kappa$ together with 
 $f^{-1}(T_\kappa\cap T^\kappa)$. Accordingly, 
 $N_G^{(\kappa)}(T^\kappa)/T^\kappa$ is a semi-direct product of 
 $W^\kappa=N_G(T)^\kappa/T^\kappa$, with the finite group 
 $T^\kappa\cap T_\kappa$. In a similar fashion, the subgroup of $L^{(\kappa)}G$ preserving $\t^\kappa\subset \A^{(\kappa)}$ is generated by $\Lambda^\kappa\rtimes N_G(T)^\kappa$, together with paths of the form $\lambda(t)=h\exp(tX)$ with 
$X\in \t^\kappa$ such that $f(h)\exp_T(X)=e$. 
Letting $\Lambda^{(\kappa)}=\exp_{T^\kappa}^{-1}(T_\kappa\cap T^\kappa)\subset \t^\kappa$, 
the resulting transformation group of $\t^\kappa$ is $\Waff^{(\kappa)}=\Lambda^{(\kappa)}\rtimes W^\kappa$. (This is the Weyl group of the twisted affine Kac-Moody algebra corresponding to $L^{(\kappa)}G$.) 
\end{remark}

\begin{appendix}

\section{Spaces of compatible complex structures}\label{app:contractible}
Recall the setup of Section \ref{subsec:compatiblecs}: $\mathsf{H}$ is a real Hilbert space with inner product $g$, (strong) symplectic structure $\omega$, and complex structure $J$ such that
\[ g(v,w)=\omega(v,Jw).\]
For $A \in \mathbb{B}(\mathsf{H})$, let $A^\ast$ denote the transpose of $A$ with respect to $g$.  The polar decomposition
\[ A=RP, \qquad P=|A|=\sqrt{A^\ast A}, \qquad R=A|A|^{-1}\]
leads to a contraction
\[ A_t:=RP^t, \qquad t \in [0,1]\]
of the space of invertible elements of $\mathbb{B}(\mathsf{H})$ onto the orthogonal group $\on{O}(\mathsf{H})$.  The following is well-known.
\begin{proposition}
The contraction $t \mapsto A_t$ restricts to a contraction of $\on{Sp}(\mathsf{H})$ onto $\on{U}_J(\mathsf{H})$.
\end{proposition}
\begin{proof}
Note that $A \in \on{Sp}(\mathsf{H})$ implies $A^\ast JA=J$.  In particular $A^\ast \in \on{Sp}(\mathsf{H})$, hence $A^\ast A \in \on{Sp}(\mathsf{H})$.  Taking fractional powers of the equation $JA^\ast AJ^{-1}=(A^\ast A)^{-1}$ shows that $|A|^t \in \on{Sp}(\mathsf{H})$ for all $t \in \mathbb{Q}$, hence for all $t \in \mathbb{R}$ by norm-closedness of $\on{Sp}(\mathsf{H})$.  Setting $t=1$ shows that $R=A|A|^{-1} \in \on{O}(\mathsf{H})\cap \on{Sp}(\mathsf{H})=\on{U}_J(\mathsf{H})$.
\end{proof}

The contraction descends to the quotient, giving a contraction of $\ca{J}(\mathsf{H},\omega)$ to a point.

\begin{proposition}
The contraction $t \mapsto A_t$ restricts to a contraction of $\on{Sp}_{res}(\mathsf{H})$ onto $\on{U}_J(\mathsf{H})$.
\end{proposition}
\begin{proof}
Let $A \in \on{Sp}_{res}(\mathsf{H})$ have polar decomposition $A=RP$.  Since $R \in \on{U}_J(\mathsf{H})$,
\begin{equation} 
\label{JCommutator}
[J,A_t]=R[J,P^t].
\end{equation}
Extend $P$ complex-linearly to $\mathsf{H}^{\mathbb{C}}$.  The spectrum of $P$ is a compact subset of $(0,\infty)$.  Choose a simple closed contour $\Gamma$ contained in $\mathbb{C} \setminus (-\infty,0]$ and containing the spectrum of $P$.  Then
\[ P^t=\frac{1}{2\pi i}\int_\Gamma z^t (z-P)^{-1}dz,\]
where $z^t=e^{t\log(z)}$ (branch cut along $(-\infty,0]$), hence
\[ [J,P^t]=\frac{1}{2\pi i}\int_\Gamma z^t(z-P)^{-1}[J,P](z-P)^{-1}dz.\]
Taking Hilbert-Schmidt norms
\[ \|[J,P^t]\|_{HS} \le \frac{1}{2\pi}\int_\Gamma |z^t|\cdot \|(z-P)^{-1}\|^2\|[J,P]\|_{HS}|dz| \]
shows that $\|[J,P^t]\|_{HS} < \infty$.  By equation \eqref{JCommutator}, $A_t \in \on{Sp}_{res}(\mathsf{H})$ for $t \in [0,1]$.  Using similar arguments one shows that $(t,A) \mapsto A_t$ is continuous with respect to the norm $\|-\|_J=\|-\|+\|[J,-]\|_{HS}$.
\end{proof}

\begin{remark}
By Kuiper's theorem, $\on{U}(\mathsf{H})$ is contractible in the norm topology.  Thus $\on{Sp}(\mathsf{H})$ and $\on{Sp}_{res}(\mathsf{H})$ are contractible.  For $\on{Sp}(\mathsf{H})$, this and related results can be found in \cite{swa:ls}.
\end{remark}

\section{Central extension of the loop group}\label{app:central}
Let $G$ be a compact, connected Lie group, with Lie algebra $\g$. Let $L_{\on{pol}}\g=\g[z,z^{-1}]\subset L\g^\C$ be the Lie algebra of Laurent loops. The choice of invariant metric $\cdot$ on $\g$ determines a central extension 
$\wh{ L_{\on{pol}}\g}=L_{\on{pol}}\g\oplus \C$, with bracket 
\[ [(\xi_1,t_1),(\xi_2,t_2)]=\Big([\xi_1,\xi_2],\ \tpi \int_{S^1} \d\xi_1\cdot \xi_2\Big).\]
If $G$ is simple and simply connected, then any invariant metric is a multiple of the \emph{basic inner product}. The multiple is a real number $k$ called the \emph{level} of the central extension. One knows that the Lie algebra extension exponentiates to a Lie group extension by $\C^\times$ if and only if $k\in\Z$. 

Let ${L_{\on{pol}}\g}={ L_{\on{pol}}\g}_+\oplus \g^{\C}\oplus {L_{\on{pol}}\g}_-$ be the triangular decomposition defined by the Fourier modes. We obtain a representation of $\Cl({L_{\on{pol}}\g})$ on 
\[\ca{R}=\Cl({L_{\on{pol}}\g})/\Cl({L_{\on{pol}}\g}){L_{\on{pol}}\g}_+\cong \Cl(\g)\otimes \wedge {L_{\on{pol}}\g}_-.\]  
The adjoint action of $L_{\on{pol}}\g$ gives a Lie algebra homomorphism $\xi\mapsto \ad_\xi,\ L_{\on{pol}}\g\to \mf{o}(L_{\on{pol}}\g)$, acting on $\Cl(L_{\on{pol}}\g)$ by derivations. It turns out to be impossible to lift to a Lie algebra morphism $L_{\on{pol}}\g\to \Cl(L_{\on{pol}} \g),\ \xi\mapsto \gamma(\xi)\in \Cl(L_{\on{pol}}\g)$ with 
\begin{equation}\label{eq:cet}
\ad_\xi(x)=[\gamma(\xi),x],\ \ \ x\in\Cl(L_{\on{pol}} \g).
\end{equation} 
However, it turns out that one can define $\gamma(\xi)$ as operators on
$\ca{R}$, satisfying \eqref{eq:cet} but with
\[ [\gamma(\xi),\gamma(\zeta)]=\gamma([\xi,\zeta])+\psi_{KP}(\xi,\zeta),\]
where $\psi_{KP}$ is the \emph{Kac-Peterson cocycle} \cite{kac:sp}
\[ \psi_{KP}(\xi,\zeta)=\hh \on{Res}_{z=0}B^{\on{kil}}\left(\frac{\p \xi}{\p
z},\zeta \right)=\f{1}{4\pi i}\int_{S^1} B^{\on{kil}}(\d \xi,\zeta)\]
and $B^{\on{kil}}(\xi_1,\xi_2)=\on{tr}(\ad_{\xi_1}\ad_{\xi_2})$ is the
Killing form on $\g$.  If $\g$ is simple, the Killing form is related to
the basic inner product by $B^{\on{kil}}=-8\pi^2 \cox B^{\on{basic}}$; hence
we obtain
\[ \psi_{KP}(\xi,\zeta)=\tpi \ \cox \int_{S^1} B^{\on{basic}}(\d \xi,\zeta).
\]
This verifies that the spin central extension of the loop group is at
level the dual Coxeter number.

\section{The 2-form $\varpi$}\label{app:2form}
The 2-form $\varpi\in \Omega^2(\P G)$ is given in terms of the evaluation map 
$\on{ev}_s\colon \P G\to G$ as follows, 
\[ \varpi=\hh \int_0^1 \Big(\on{ev}_t^*\theta^R\cdot \f{\p}{\p t}  \on{ev}_t^*\theta^R \Big)\partial t
+ \hh \on{ev}_0^*\theta^L\cdot \on{ev}_1^*\theta^L
\]
Letting $\A_\lambda$ denote the action of $\lambda\in LG$ on $\P G$, we have 
\[ \on{ev}_t\circ \A_\lambda=r_{\lambda(t)^{-1}}\circ \on{ev}_t,\]
where $r_a$ denotes right multiplication by $a\in G$. It is thus follows that the 
$\g$-valued 1-forms $\on{ev}_t^*\theta^R$ are $LG$-invariant: 
\[ \A_\lambda^* \on{ev}_t^*\theta^R=\on{ev}_t^*\,r_{\lambda(t)^{-1}}^*\theta^R
=\on{ev}_t^*\theta^R.\]
On the other hand, under the action of $G$ by left multiplication, 
$\on{ev}_t^*\theta^R$ transforms by the adjoint action. Hence $\varpi$ is $LG\times G$-invariant. The exterior differential of $\varpi$ is calculated by integration by parts,
similar to \cite[Appendix A]{al:mom}  and is given by 
\[ \d\varpi=\ev_0^*\eta-\ev_1^*\eta+\hh \d \big(\on{ev}_0^*\theta^L\cdot \on{ev}_1^*\theta^L\big)=-q^*\eta,\]
where we used $q(\gamma)=\gamma_1\gamma_0^{-1}$. 
Since $\on{ev}_t$ intertwines the left-action on $\PG$ with that on $G$, we have that
$X_{\P G}\sim_{\ev_t}-X^R$. 
Consequently, at $\gamma\in \PG$,
\begin{align*}
\iota(X_{\PG})\varpi&=-\hh X\cdot(\gamma_1^*\theta^R-\gamma_0^*\theta^R)
-\hh \on{Ad}_{\gamma_0^{-1}}X\cdot \gamma_1^*\theta^L
+\hh \on{Ad}_{\gamma_1^{-1}}X\cdot \gamma_0^*\theta^L\\
&=\hh X\cdot\big( (\gamma_1\gamma_0^{-1})^*\theta^L-
(\gamma_0\gamma_1^{-1})^*\theta^L
   \big)\\
&=-q^* \big(\hh X\cdot (\theta^L+\theta^R)\big).    
\end{align*}
On the other hand, for $\xi\in L\g$ we have that $\xi_{\P G}\sim_{\ev_t}\xi_t^L$. 
Therefore, at $\gamma\in \PG$,
\begin{align*}
\iota(\xi_{\P G})\varpi&=\hh \int_0^1 \Big(\Ad_{\gamma_t}\xi_t\cdot \f{\p}{\p t} \gamma_t^*\theta^R-\gamma_t^*\theta^R\cdot \f{\p}{\p t} \Ad_{\gamma_t}\xi_t\Big)\partial t +\hh \xi(0)\cdot\big( \gamma_1^*\theta^L-\gamma_0^*\theta^L\big)
\\
&=\int_0^1 \Big(\xi_t\cdot \Ad_{\gamma_t^{-1}}\f{\p}{\p t} \gamma_t^*\theta^R\Big)\partial t= p^* \d \l\mu,\xi\r
\end{align*}
where the last equality used
\[ \Ad_{\gamma_t^{-1}}\f{\p}{\p t} \gamma_t^*\theta^R=
\d(\gamma_t^{-1}\dot{\gamma}_t ).\]
In the presence of an automorphism $\kappa$, we define $\varpi^{(\kappa)}\in 
\Omega^2(\P^{(\kappa)}G)$ as follows, 
\[ \varpi^{(\kappa)}=\hh \int_0^1 \Big(\on{ev}_t^*\theta^R\cdot \f{\p}{\p t}  \on{ev}_t^*\theta^R\Big)\partial t
+ \hh \on{ev}_0^*\kappa(\theta^L)\cdot \on{ev}_1^*\theta^L.
\]
A calculation similar to the above shows $\d\varpi^{(\kappa)}=-q^*\eta$, and 
\[ \iota(X_{\PG})\varpi^{(\kappa)}=-q^* \big(\hh 
(\kappa(X)\cdot \theta^L+X\cdot \theta^R)\big),\ \ \ 
\iota(\xi_{\PG})\varpi^{(\kappa)}=p^*\l\d\mu,\xi\r
.    \]
\end{appendix}

\def\cprime{$'$} \def\polhk#1{\setbox0=\hbox{#1}{\ooalign{\hidewidth
  \lower1.5ex\hbox{`}\hidewidth\crcr\unhbox0}}} \def\cprime{$'$}
  \def\cprime{$'$} \def\cprime{$'$} \def\cprime{$'$} \def\cprime{$'$}
  \def\polhk#1{\setbox0=\hbox{#1}{\ooalign{\hidewidth
  \lower1.5ex\hbox{`}\hidewidth\crcr\unhbox0}}} \def\cprime{$'$}
  \def\cprime{$'$} \def\cprime{$'$} \def\cprime{$'$} \def\cprime{$'$}
\providecommand{\bysame}{\leavevmode\hbox to3em{\hrulefill}\thinspace}
\providecommand{\MR}{\relax\ifhmode\unskip\space\fi MR }
\providecommand{\MRhref}[2]{%
  \href{http://www.ams.org/mathscinet-getitem?mr=#1}{#2}
}
\providecommand{\href}[2]{#2}

\end{document}